\documentclass[12pt]{amsart}
\usepackage{a4wide,enumerate,xcolor}
\usepackage{amsmath,graphicx,bm}
\usepackage{mathtools} 
\allowdisplaybreaks

\usepackage[normalem]{ulem}

\usepackage{enumitem}
\setlist[itemize]{label={$\bullet$}, leftmargin=32pt, itemsep=3pt}

\let\pa\partial
\let\na\nabla
\let\eps\varepsilon
\newcommand{\N}{{\mathbb N}}
\newcommand{\R}{{\mathbb R}}
\newcommand{\diver}{\operatorname{div}}

\newcommand{\E}{\mathcal{E}}
\newcommand{\T}{\mathcal{T}}
\newcommand{\dom}{\mathcal{D}}
\newcommand{\m}{\mathrm{m}}
\newcommand{\D}{\mathrm{D}}
\newcommand{\dist}{\mathrm{d}}

\newtheorem{theorem}{Theorem}
\newtheorem{lemma}[theorem]{Lemma}
\newtheorem{proposition}[theorem]{Proposition}
\newtheorem{remark}[theorem]{Remark}

\begin{document}

\title[Finite-volume scheme for two-phase flow models]{
Energy dissipation and stability of a \\
finite-volume scheme for two-phase flow models \\
with dynamic capillary pressure}

\author[A. J\"ungel]{Ansgar J\"ungel}
\address{Institute of Analysis and Scientific Computing, TU Wien, Wiedner Hauptstra\ss e 8--10, 1040 Wien, Austria}
\email{juengel@tuwien.ac.at}

\author[J.-P. Mili\v{s}i\'c]{Josipa-Pina Mili\v{s}i\'c}
\address{University of Zagreb, Faculty of Electrical Engineering and Computing, Unska 3, 10000 Zagreb, Croatia}
\email{pina.milisic@fer.unizg.hr}

\author[S. Xhahysa]{Sara Xhahysa}
\address{Institute of Analysis and Scientific Computing, TU Wien, Wiedner Hauptstra\ss e 8--10, 1040 Wien, Austria}
\email{sara.xhahysa@tuwien.ac.at}

\date{\today}

\thanks{The first and third authors acknowledge partial support from the Austrian Science Fund (FWF), grant 10.55776/PAT2687825, and from the Austrian Federal Ministry for Women, Science and Research and implemented by \"OAD, project MULT09/2025. This work has received funding from the European Research Council (ERC) under the European Union's Horizon 2020 research and innovation programme, ERC Advanced Grant NEUROMORPH, no.\ 101018153. The second author acknowledges the support of the Republic of Croatia’s MSEY in the course of multilateral scientific and technological cooperation in the Danube region under the project MultiHeFlo and Croatian Science Foundation under the project IP-2025-02-5004. For open-access purposes, the authors have applied a CC BY public copyright license to any author-accepted manuscript version arising from this submission.}

\begin{abstract}
An implicit Euler finite-volume scheme for degenerate pseudo-parabolic cross-diffusion equations is proposed and analyzed. The system describes the dynamics of an unsaturated two-phase flow mixture with dynamic capillary pressure in a porous medium. The numerical scheme is based on a two-point flux approximation that preserves the energy structure, ensures the conservation of total mass, and guarantees strict positivity and boundedness of the water saturation. These properties rely on carefully selected mean functions for the nonlinear coefficients. The existence of discrete solutions, their uniqueness among nondegenerate solutions for sufficiently small time steps, and additional mesh-uniform bounds are proved. Numerical experiments in two space dimensions illustrate the effects of the dynamic capillary pressure.
\end{abstract}



\maketitle


\section{Introduction}

The modeling of chemical mixture transport in porous media is of central importance in many applications like groundwater remediation, enhanced oil recovery, and carbon dioxide sequestration. Such processes can be described by two-phase flow models. Classical formulations assume an instantaneous capillary equilibrium between both phases, but experimental evidence shows that capillary pressure can exhibit dynamic effects, depending on the rate of saturation change. Incorporating dynamic capillary pressure introduces an additional relaxation term that captures nonequilibrium interfacial effects more accurately. Moreover, many practical systems involve multiple chemical species within each phase, necessitating the inclusion of multicomponent diffusion to account for mass transfer driven by concentration gradients as well as cross-diffusion effects. The resulting model consists of coupled mass balance equations with cross-diffusion and advective fluxes driven by pressure gradients and derived from a Richards-type flow formulation.

In this paper, we discretize the resulting thermodynamically consistent pseudo-parabolic cross-diffusion system in a bounded domain using a two-point flux approximation finite-volume scheme. We prove that the scheme preserves the energy structure, the positivity and upper bound of the water saturation, as well as the total mixture mass. Moreover, we show the stability of the numerical solution and present some numerical experiments in two space dimensions.

\subsection{Model setting}

We consider an incompressible, immiscible two-phase fluid mixture in a bounded domain $\Omega\subset\R^d$ ($d\ge 1$). The two-phase fluid, e.g.\ composed of water and gas, is described by the Richards model \cite{ChJa86}, where the water is a mixture of $n$ components. The gas is assumed to be homogeneous, perfectly mobile, and at constant pressure, so it can be removed from the system. Let $\rho_i\ge 0$ be the partial densities of the water components, $v_i\in\R^d$ their associated velocities, and $S\in[0,1]$ be the water saturation. Each component of the water mixture satisfies the mass balance equation
\begin{align*}
  \Phi(x)\pa_t(S\rho_i) + \diver(\rho_iv_i) = 0, \quad i=1,\ldots,n,
\end{align*}
where $\Phi(x)$ is the given porosity. Let $\rho=\sum_{i=1}^n\rho_i$ be the (constant) total mass density and $v$ be the barycentric velocity, defined by $\rho v=\sum_{i=1}^n\rho_iv_i$. Introducing the diffusive flux $J_i=\rho_i(v_i-v)$, the partial flux can be written as $\rho_iv_i = \rho_i v+J_i$. In the following, we propose models for $J_i$ and $v$.

The diffusion flux is given in the Fick--Onsager formulation as
\begin{align}\label{1.Ji}
  J_i = -\sum_{j=1}^n M_{ij}\na\mu_j, \quad i=1,\ldots,n,
\end{align}
where $(M_{ij})$ is the mobility matrix satisfying $\sum_{i=1}^n M_{ij} = 0$, and the chemical potentials $\mu_j$ are determined from the free energy (see below). By Onsager's reciprocal relations, the matrix $(M_{ij})$ is symmetric, and by the second law of thermodynamics, it is positive semidefinite. The definition of the flux $J_i$ shows that $\sum_{i=1}^n J_i=0$, and the condition $\sum_{i=1}^n M_{ij} = 0$ is consistent with this property. It implies that the mobility matrix has a vanishing eigenvalue. Thus, we cannot expect coercivity of the associated bilinear form but only coercivity on a subspace:
\begin{align*}
  \sum_{i,j=1}^nM_{ij}z_iz_j\ge c_M|\Pi z|^2
  \quad\mbox{for all }z\in\R^n
\end{align*}
and some $c_M>0$, where $\Pi$ is the orthogonal projection onto the orthogonal complement of $\mbox{span}\{\mathrm{1}\}$ with $\mathrm{1}=(1,\ldots,1)^T$.

The barycentric velocity is often given by Darcy's law $v=-a(S)\na p$ with the diffusion mobility $a(S)$ and the pressure $p$. The mixture pressure $p$ is related to the capillary pressure $p_c(S)$ by $p_c(S)=p_{\rm at}-p$, where $p_{\rm at}$ is the constant atmospheric pressure \cite{Bea88}. This description assumes an instantaneous capillary equilibrium. In situations where one phase rapidly displaces another, such as in water flooding (oil displacement), the capillary pressure also depends on the rate of change of the saturation. Then the model is replaced by $p_{\rm at}-p = p_c(S)-b(S)\pa_t S$, where $b(S)$ denotes the dynamic capillary pressure coefficient \cite{HaGr93}. Introducing the function $\beta(S)$ by $\beta'(S)=b(S)$ and setting $P_c(S):=-p_c(S)$, the barycentric velocity is then given by
\begin{align}\label{1.v}
  v = -a(S)\na p = -a(S)\na(P_c(S)+\pa_t\beta(S)).
\end{align}
Inserting \eqref{1.Ji} and \eqref{1.v} into the expression for the partial fluxes leads to
\begin{align*}
  \rho_i v_i = -\rho_i a(S)\na \big(P_c(S)+\pa_t\beta(S)\big)
  - \sum_{j=1}^n M_{ij}\na\mu_j, \quad i=1,\ldots,n.
\end{align*}

We introduce the variables $S_i:=\rho_i S$ and normalize the total mass density such that $\sum_{i=1}^n\rho_i=1$. Then $S=\sum_{i=1}^n S_i$, and we can write the equations as
\begin{align}\label{1.Si}
  & \Phi(x)\pa_t S_i + \diver F_i = 0\quad\mbox{in }\Omega,\ t>0,\
  i=1,\ldots,n, \\
  & F_i = -\frac{S_i}{S}a(S)\na\big(P_c(S) + \pa_t\beta(S)\big)
  - \sum_{j=1}^n M_{ij}(\bm{S})\na\mu_j, \label{1.Fi}
\end{align}
where $\bm{S}=(S_1,\ldots,S_n)$. We impose the initial and mixed boundary conditions
\begin{align}\label{1.bic}
  S_i(0)=S_i^0\quad\mbox{in }\Omega, \quad
  S_i = S_i^D\quad\mbox{on }\Gamma_D, \quad
  F_i\cdot\nu = 0\quad\mbox{on }\Gamma_N,\ t>0,
\end{align}
for $i=1,\ldots,n$, where $\pa\Omega=\Gamma_D\cup\Gamma_N$, $\Gamma_D\cap\Gamma_N=\emptyset$, the measure of $\Gamma_D$ is positive, and $\Gamma_N$ is open in $\pa\Omega$. To be physically meaningful, we require that $S_i\ge 0$ and $S\in[0,1]$. We set $S^D=\sum_{i=1}^n S_i^D$.

To complete the model equations, we need to specify the chemical potentials. For this, we introduce the free energy $E$, consisting of the relative bulk and capillary energy densities:
\begin{align}\label{1.E}
  & E(\bm{S}) = E_{\rm bulk}(\bm{S}) + \Psi(S),
  \quad\mbox{where} \\
  & E_{\rm bulk}(\bm{S}) = \sum_{i=1}^n S_i\log\frac{S_i}{S_i^D}
  - S\log\frac{S}{S^D}, \quad
  \Psi(S) = \int_{S^D}^S\int_{S^D}^z
  \frac{b(y)}{a(y)}dydz. \nonumber
\end{align}
This gives the chemical potential
\begin{align}\label{1.mui}
  \mu_i = \frac{\pa E}{\pa S_i} = \log\frac{S_i}{S}
  - \log\frac{S_i^D}{S^D} + \psi(S),
  \quad\mbox{where }\psi(S) = \int_{S^D}^S\frac{b(y)}{a(y)}dy.
\end{align}
The choice of $E$ is motivated by the Gibbs--Duhem relation and a convexity requirement \cite[Remark 1]{DMZ20}.

Our model consists of equations \eqref{1.Si}--\eqref{1.bic}, \eqref{1.mui}. Mathematically, \eqref{1.Si}--\eqref{1.Fi} is a system of pseudo-parabolic cross-diffusion equations. They were analyzed in \cite{DMZ20}, and the existence of a global weak solution was shown. The aim of this paper is the design of a structure-preserving finite-volume scheme for \eqref{1.Si}--\eqref{1.bic}, \eqref{1.mui} and the proof of uniform stability bounds.


\subsection{State of the art}

Two-phase flow in porous media is a classical topic with applications in hydrology, petroleum engineering, environmental sciences, and industrial processes. In the standard Richards-type formulation, the capillary pressure is usually assumed to be an equilibrium function of the saturation. However, experimental and theoretical studies have shown that this assumption may be insufficient in situations where nonequilibrium effects are relevant. This has motivated the introduction of dynamic capillary pressure models, in which the capillary pressure depends not only on the saturation but also on its time derivative.

A rigorous mathematical analysis of two-phase flow models with dynamic capillary pressure was developed by Koch, R\"atz, and Schweizer \cite{KoRaSch13}. 
On the numerical side, finite-volume methods have become one of the standard discretization techniques for two-phase flow problems in porous media due to their local conservation properties and robustness in the presence of strongly heterogeneous coefficients. In particular, Canc\`es \cite{Can09} developed a convergent finite-volume scheme for incompressible two-phase flow in heterogeneous porous media with discontinuous capillary pressure laws. The proposed method preserves the conservative structure of the model and was shown to converge towards a weak solution of the continuous problem. Subsequently, Brenner, Canc\`es, and Hilhorst \cite{BreCanHil13} extended these ideas to multidimensional heterogeneous media and analyzed finite-volume approximations incorporating nonlinear transmission conditions across material interfaces. These works established a mathematical foundation for finite-volume discretizations of degenerate two-phase flow systems and inspired structure-preserving methods for more complex models.

Numerical aspects were further investigated by Zhang and Zegeling \cite{ZhZeg17}, who studied two-phase flow models including both dynamic capillary pressure and hysteresis. A finite-volume approximation for incompressible two-phase flow with dynamic capillary pressure was analyzed by Bouadjila, Mokrane, Saad, and Saad \cite{BoMoSaSa18}, providing convergence and stability results for a fully discrete scheme.

More recently, structure-preserving numerical methods for two-phase flow have attracted considerable attention. Kou, Chen, Salama, and Sun \cite{KoChSaSu24} proposed an energy-stable and positivity-preserving computational method for compressible and immiscible two-phase flow in porous media.  Related energy-stable modeling ideas were developed by Kou, Du, and Zhong \cite{KoDuZh21}, who incorporated fluid--fluid friction forces using a Maxwell--Stefan--Darcy approach.


The present paper contributes to this line of research by studying a multicomponent incompressible two-phase flow model of Richards type with cross-diffusion and dynamic capillary pressure. Compared with the existing literature, the main novelty lies in the combination of the energy structure, the treatment of multicomponent cross-diffusion effects, and the construction of a finite-volume scheme that is compatible with the underlying dissipative structure of the continuous model.


\subsection{Main results}

Our first result is the derivation of a discrete analog, for the finite-volume solutions, of the energy inequality
\begin{align}\label{1.ei}
  \frac{dH}{dt}(\bm{S})
  + c_M\int_\Omega|\Pi\na\bm{\mu}|^2 dx
  + \int_\Omega\beta'(S)P'_c(S)|\na S|^2 dx \le 0,
\end{align}
where the total energy $H$ is the sum of the free energy and the interfacial energy,
\begin{align*}
  H(\bm{S}) = \int_\Omega\bigg(\Phi(x)E(\bm{S})
  + \frac12|\na\beta(S)|^2\bigg)dx;
\end{align*}
see Appendix \ref{sec.app}. To make this statement more explicit, let $(\bm{S}_K^k)_{K\in\T}$ be a piecewise constant approximation of $\bm{S}$ at time $t_k=k\Delta t$, where $\Delta t>0$ is the time step and $K\in\T$ are the control volumes of the finite-volume scheme. We denote by $\D_{K,\sigma}\beta(S^k)$ a difference approximation of $\na\beta(S)$ on the edge $\sigma$ of the control volume $K$ at time $t_k$. We refer to Section \ref{sec.num} for details on the notation. We show in Theorem \ref{thm.dei} that the discrete total energy
\begin{align}\label{1.Hd}
  H_d(\bm{S}^k) = \sum_{K\in\T}\m(K)\Phi_K E(\bm{S}_K^k)
  + \frac12\sum_{\sigma\in\E}\tau_\sigma|\D_{K,\sigma}\beta(S^k)|^2,
\end{align}
where $S_K^k=\sum_{i=1}^n S_{i,K}^k$ and $\tau_\sigma$ is the so-called transmissibility coefficient, defined in \eqref{2.trans} below, satisfies the discrete energy inequality
\begin{align}\label{1.dei}
  H_d(\bm{S}^k) + c_M\Delta t\sum_{\sigma\in\E}\tau_\sigma
  |\Pi(\D_{K,\sigma}\bm{\mu}^k)|^2
  + \Delta t\sum_{\sigma\in\E}\tau_\sigma
  \D_{K,\sigma}P_c(S^k)\D_{K,\sigma}\beta(S^k)
  \le H_d(\bm{S}^{k-1}),
\end{align}
where $\bm\mu^k$ is the discrete chemical potential vector. The proof of \eqref{1.ei} relies on nonlinear chain rules, which must be suitably adapted to the discrete setting in order to derive \eqref{1.dei}. The key challenge lies in the term $(S_i/S)a(S)$, which requires an approximation on each edge of the control volumes. This is accomplished by using the logarithmic mean and taking advantage of the fact that $a(S) = b(S)/\psi'(S) =\beta'(S)/\psi'(S)$ (see \eqref{1.E}--\eqref{1.mui}). This expression is then approximated using finite differences.

Besides the structure-preservation property, we establish the following results:
\begin{itemize}
\item Existence of a discrete solution (Theorem \ref{thm.ex}): The proof is based on the bounds from the discrete energy inequality and Schaefer's fixed-point theorem. The singular structure of $\Psi(S)$ at $S=0$ and $S=1$ (see \eqref{1.E} and \eqref{2.Psi}) implies that there exists $\kappa\in(0,1)$ such that $\kappa\le S_{K}^k\le 1-\kappa$, which prevents the degeneracy of the function $a(S)$. By exploiting properties of the logarithm, we are also able to show that $S_{i,K}^k>0$.
\item Uniqueness of a solution (Theorem \ref{thm.unique}): For given $\kappa\in(0,1/2)$, there exists at most one discrete solution satisfying the non-degeneracy property $S_{i,K}^k\ge\kappa$, $S_K^k\le 1-\kappa$, provided that $\Delta t$ is sufficiently small (depending on $\kappa$ and the mesh). The dependence on the mesh parameters is discussed in Remark \ref{rem.mesh}.
\item Stability estimates (Proposition \ref{prop.beta}): Additional estimates can be derived from the inequality
\begin{align}\label{1.beta}
  \int_0^T\int_\Omega|\pa_t\beta(S)|^2 dx dt
  + \int_0^T\int_\Omega a(S)|\na\pa_t\beta(S)|^2 dx dt \le C(\bm{S}^0)
\end{align}
(see Appendix \ref{sec.app}), which is translated to the discrete level. Under some conditions on the nonlinearities, we show that
\begin{align*}
  \Delta t \sum_{k=1}^N\sum_{K\in\T}
  \m(K)\bigg|\frac{\beta(S_K^k)-\beta(S_K^{k-1})}{\Delta t}\bigg|^2
  + \Delta t \sum_{k=1}^N\sum_{\sigma\in\E}\tau_\sigma a_\sigma^k
  \bigg|\D_{K,\sigma}\bigg(\frac{\beta(S^k)-\beta(S^{k-1})}{\Delta t}
  \bigg)\bigg|^2 \le C,
\end{align*}
where $a_\sigma^k$ approximates $a(S)$ on the edge $\sigma$ at time $t_k$.
\end{itemize}

Furthermore, we present some numerical experiments in two space dimensions.


\subsection{Assumptions}

We specify our assumptions on the nonlinear functions. To avoid technical difficulties, we assume explicit forms for $a$, $b$, and $P_c$ as in \cite{Mil18}.

\begin{itemize}
\item[\bf (A1)] The porosity $\Phi:\Omega\to[0,1]$ is integrable and satisfies $\Phi\ge c_\Phi>0$ in $\Omega$. The initial data satisfy $H(\bm{S}^0)<\infty$ and $S_i^0\ge c$, $S^0=\sum_{i=1}^n S_i^0\le 1-c$ in $\Omega$ for some $c\in(0,1)$. The boundary data $S_i^D>0$ are constant, and $S^D=\sum_{i=1}^n S_i^D<1$.
\item[\bf (A2)] The mobility matrix $(M_{ij})\in C^1([0,1]^n;\R^{n\times n})$ is symmetric, positive semidefinite, and there exist $0<c_M<C_M<\infty$ such that
\begin{align*}
  c_M|\Pi z|^2 \le \sum_{i,j=1}^n M_{ij}(\bm{S})z_iz_j
  \le C_M|\Pi z|^2 \quad\mbox{for all }z\in\R^n,\ \bm{S}\in\dom,
\end{align*}
where {$\Pi=I-(\mathrm{1}\otimes\mathrm{1})/n$} is the orthogonal projection on $\{z\in\R^n:z\cdot\mathrm{1}=0\}$ with $\mathrm{1}=(1,\ldots,1)\in\R^n$ and $\dom=\{\bm{S}\in\R^n:S_1,\ldots,S_n>0$, $\sum_{i=1}^n S_i<1\}$.
\item[\bf (A3)] The diffusion mobility $a(S)$ and the relaxation coefficient $b(S)=\beta'(S)$ equal
\begin{align*}
  a(S) = \frac{S^{\gamma_0}(1-S)^{\gamma_1}}{S^{\gamma_0}
  + (1-S)^{\gamma_1}}, \quad
  b(S) = \frac{S^{\gamma_0}}{S^{\gamma_0}+(1-S)^{\gamma_1}}
  \bigg(1 + \frac{(1-S)^{\gamma_1}}{S^{\gamma_2}}\bigg)
\end{align*}
for $S\in[0,1]$, where $\gamma_0>\gamma_2$, $\gamma_1>2$, and $\gamma_2>2$.
\item[\bf (A4)] The stationary capillary pressure $P_c(S)$ is defined by
\begin{align*}
  P'_c(S) = \frac{1}{S^{p_0}} + \frac{1}{(1-S)^{p_1}},
  \quad\mbox{where }1\le p_0\le\gamma_2,\, 1\le p_1\le\gamma_1.
\end{align*}
\end{itemize}

Assumption (A1) is needed to obtain a parabolic problem. Assumption (A2) is the strongest coercivity condition that the mobility matrix can satisfy under the constraint $\sum_{i=1}^n J_i=0$. The choice $z_i=1$ in Assumption (A2) shows, together with the symmetry and positive semidefiniteness of $(M_{ij})$, that $\sum_{i=1}^n M_{ij}=0$ for $j=1,\ldots,n$. The function $a(S)$ is interpreted as the mobility of the wetting phase. Physically, it should be nonnegative and satisfy the degeneracy conditions $a(0)=a(1)=0$. The relaxation coefficient $b(S)$ should be nonnegative, increasing, and vanish at $S=0$ (see Lemma \ref{lem.beta} below). These properties are satisfied by the choice of $a(S)$ and $b(S)$ in Assumption (A3). The stationary capillary pressure is increasing and satisfies $P_c(0)=-\infty$, $P_c(1)=\infty$ since $p_0\ge 1$ and $p_1\ge 1$. Furthermore, a computation gives
\begin{align*}
  \psi'(S) = \frac{b(S)}{a(S)}
  = \frac{1}{S^{\gamma_2}} + \frac{1}{(1-S)^{\gamma_1}}.
\end{align*}
It follows from definition \eqref{1.E} of $\Psi$ and integration from $S^D$ to $S$ that
\begin{align}\label{2.Psi}
  \Psi(S) &= \frac{(1-S)^{2-\gamma_1}-(1-S^D)^{2-\gamma_1}}{
  (\gamma_1-1)(\gamma_1-2)}
  +\frac{S^{2-\gamma_2}-(S^D)^{2-\gamma_2}}{(\gamma_2-1)(\gamma_2-2)} \\
  &\phantom{xx}- \frac{(1-S^D)^{1-\gamma_1}}{\gamma_1-1}(S-S^D)
  + \frac{(S^D)^{1-\gamma_2}}{\gamma_2-1}(S-S^D). \nonumber
\end{align}
Thus, $\Psi(0)=\Psi(1)=\infty$, and we conclude from the energy inequality \eqref{1.ei} and $\gamma_1$, $\gamma_2>2$ that $0<S<1$ a.e.\ in $\Omega$. This property will be extended to the discrete case, giving in fact the stronger property $0<\kappa\le S_K^k\le 1-\kappa$ for some $\kappa\in(0,1)$. We infer from the conditions $p_0\le\gamma_2$ and $p_1\le\gamma_1$ in Assumption (A4) that there exists $C>0$ such that
\begin{align*}
  P'_c(S)\le C\frac{b(S)}{a(S)} = C\psi'(S) \quad\mbox{for }0<S<1.
\end{align*}
This property is used to derive the additional bounds \eqref{1.beta}. On the discrete level, the derivatives are replaced by finite differences; see the proof of Proposition \ref{prop.beta}.


\subsection{Organization of the paper}

The numerical scheme is introduced in Section \ref{sec.num}. We prove the discrete energy inequality in Section \ref{sec.dei}, which is used to show the existence of a discrete solution in Section \ref{sec.ex}. The uniqueness of the solution is proved in Section \ref{sec.unique}. In Section \ref{sec.stab}, we establish additional discrete stability bounds and discuss the convergence of the scheme. Some numerical experiments are presented in Section \ref{sec.exp}. Finally, inequalities \eqref{1.ei} and \eqref{1.beta} are shown in Appendix \ref{sec.app}.


\section{Numerical scheme}\label{sec.num}

We introduce the notation needed for our finite-volume scheme and present the numerical scheme.

\subsection{Notation}

The domain $\Omega$ is discretized by an admissible triangulation in the sense of \cite[Definition 9.1]{EGH00}. The triangulation consists of a family $\T$ of open polygonal convex subsets of $\Omega$ (cells or control volumes), a family $\E$ of edges (or faces in three space dimensions), and a family of points $\mathcal{P}=(x_K)_{K\in\T}$ associated to the cells. The admissibility assumption implies that the line segment connecting the points $x_K$ and $x_L$ of two neighboring cells is orthogonal to their common edge $\sigma=K|L$.

The family of edges $\E$ is split into internal and external edges $\E=\E_{\rm int}\cup\E_{\rm ext}$, where $\E_{\rm int} = \{\sigma\in\E:\sigma\subset\Omega\}$ and $\E_{\rm ext} = \{\sigma\in\E:\sigma\subset\pa\Omega\}$. Each exterior edge is assumed to be an element of either the Dirichlet or Neumann boundary, setting $\E_{\rm ext} = \E_{\rm ext}^D\cup\E_{\rm ext}^N$. For given $K\in\T$, we define the set $\E_K$ of the edges of $K$, which consists of internal edges and possibly of edges on the Dirichlet or Neumann boundary, $\E_K= \E_{{\rm int},K}\cup\E_{{\rm ext},K}^D\cup\E_{{\rm ext},K}^N$.

The size of the mesh is denoted by $\Delta x = \sup\{\mbox{diam}(K):K\in\T\}$. For given $\sigma\in\E$, we define
\begin{align*}
  \dist_\sigma = \begin{cases}
  d(x_K,x_L) &\mbox{if }\sigma=K|L\in\E_{\rm int}, \\
  d(x_K,\sigma) &\mbox{if }\sigma\in\E_{\rm ext},
  \end{cases}
\end{align*}
where $d(x,y)$ is the Euclidean distance between two points $x$ and $y$. The transmissibility coefficient is defined by
\begin{align}\label{2.trans}
  \tau_\sigma = \frac{\m(\sigma)}{\dist_\sigma},
\end{align}
where $\m(\sigma)$ denotes the Lebesgue measure of $\sigma$.

We use a uniform time discretization with time step $\Delta t>0$, and we set $t_k=k\Delta t$ for $k=0,\ldots,N$, where $T>0$, $N\in\N$, and $\Delta t=T/N$. We denote by $\mathcal{M}$ an admissible space-time discretization of $\Omega\times(0,T)$ composed of an admissible mesh $(\T,\E,\mathcal{P})$ and the values $(\Delta x,\Delta t)$. The size of $\mathcal{M}$ is given by $\eta:=\max\{\Delta x,\Delta t\}$.

We denote by $V_{\Delta x}^n$ the space of piecewise constant vector-valued functions, defined by
\begin{align*}
  V_{\Delta x}^n = \bigg\{v:\overline\Omega\to\R^n: \exists
  (v_K)_{K\in\T}\subset\R^n,\,
  v(x) = \sum_{K\in\T}v_K\mathrm{1}_K(x)\bigg\},
\end{align*}
where $\mathrm{1}_K$ is the characteristic function of $K$. Furthermore, we define for $v\in V_{\Delta x}^n$ and $\sigma\in\E$ the difference
\begin{align*}
  \D_{K,\sigma}v = v_{K,\sigma}-v_K, \quad\mbox{where}\quad
  v_{K,\sigma} = \begin{cases}
  v_L &\mbox{if }\sigma=K|L\in\E_{{\rm int},K}, \\
  v_\sigma^D &\mbox{if }\sigma\in\E_{{\rm ext},K}^D, \\
  v_K &\mbox{if }\sigma\in\E_{{\rm ext},K}^N,
  \end{cases}
\end{align*}
where $v_\sigma^D$ is a given boundary value on $\sigma$ (for instance, $v_\sigma^D=\m(\sigma)^{-1}\int_\sigma v^D ds$ for a given Dirichlet datum $v^D$). We introduce the discrete $H^1(\Omega)$ seminorm and discrete $H^1(\Omega)$ norm as
\begin{align*}
  |v|_{1,2,\T} = \bigg(\sum_{\sigma\in\E}\tau_\sigma
  |\D_{K,\sigma}v|^2\bigg)^{1/2}, \quad
  \|v\|_{1,2,\T} = \big(\|v\|_{0,2,\T}^2 + |v|_{1,2,\T}^2\big)^{1/2},
\end{align*}
respectively, where the $L^p(\Omega)$ norm for $1\le p<\infty$ is given by
\begin{align*}
  \|v\|_{0,p,\T} = \bigg(\sum_{K\in\T}\m(K)|v_K|^p\bigg)^{1/p}.
\end{align*}
The discrete integration-by-parts formula becomes for $v\in V_{\Delta x}^n$ and some conservative numerical flux $F_{K,\sigma}$ (i.e., $F_{K,\sigma}=-F_{L,\sigma}$ for $\sigma=K|L\in\E_{\rm int}$ and $F_{K,\sigma}=0$ for $\sigma\in\E_{\rm ext}^N$):
\begin{align}\label{2.dibp}
  \sum_{K\in\T}\sum_{\sigma\in\E_K}F_{K,\sigma}v_K
  = -\sum_{\sigma\in\E}F_{K,\sigma}\D_{K,\sigma}v
  + \sum_{\sigma\in\E_{\rm ext}^D}F_{K,\sigma}v_{K,\sigma}.
\end{align}


\subsection{Numerical scheme}

The implicit Euler finite-volume scheme now reads as follows:
\begin{align}
  & \frac{\m(K)}{\Delta t}\Phi_K(S_{i,K}^k - S_{i,K}^{k-1})
  + \sum_{\sigma\in\E_K}F_{i,K,\sigma}(\bm{S}^k) = 0,
  \quad K\in\T,\ k\in\N,\
  i=1,\ldots,n, \label{2.eq1} \\
  & F_{i,K,\sigma}(\bm{S}^k) = -\tau_\sigma
  c_{i,\sigma}^k a_{\sigma}^k\D_{K,\sigma}\bigg(P_c(S^k)
  + \frac{\beta(S^k)-\beta(S^{k-1})}{\Delta t}\bigg)
  \label{2.eq2}
  - \tau_\sigma\sum_{j=1}^n M_{ij}(\bm{S}_\sigma^k)
  \D_{K,\sigma}\mu_j^k, 
\end{align}
with the initial and boundary conditions
\begin{align*}
  S_{i,K}^0 = \frac{1}{\m(K)}\int_K S_i^0dx, \quad
  S_{i,\sigma}^D = \frac{1}{\m(\sigma)}\int_\sigma S_i^D ds
  \quad\mbox{for }K\in\T,\ \sigma\in\E_{\rm ext}^D,
\end{align*}
and $i=1,\ldots,n$. The Dirichlet values of the chemical potentials are $\mu_{i,\sigma}^D=0$, which is consistent with \eqref{1.mui}, since $S_i^D$ is constant. In equations \eqref{2.eq1}--\eqref{2.eq2}, we have used the following notation:
\begin{align*}
  \Phi_K &= \frac{1}{\m(K)}\int_K\Phi(x)dx, \quad
  a_{\sigma}^k = \frac{\beta(S_{K,\sigma}^k)-\beta(S_{K}^k)
  }{\psi(S_{K,\sigma}^k)-\psi(S_K^k)}, \quad
  c_{i,K}^k = \frac{S_{i,K}^k}{S_K^k},
\end{align*}
where the values $S_{i,K}^k$ are given by
\begin{align*}
  S_{i,K}^k = S_i(\bm{\mu}_K^k)\quad\mbox{for }
  K\in\T,\,k\in\N,\ i=1,\ldots,n.
\end{align*}
The function $\bm{S}=(S_1,\ldots,S_n)^T:\R^n\to\dom$, $\bm \mu\mapsto\bm{S}(\bm{\mu})$,
is the inverse of
\begin{align*}
  \mu_i = \log(S_i/S) - \log(S_i^D/S^D) + \psi(S), \quad i=1,\ldots,n,
\end{align*}
which exists thanks to \cite[Lemma 4.1]{DMZ20}.
(Recall that $\dom=\{\bm{S}\in\R^n:S_1,\ldots,S_n>0$, $0< S<1\}$.)

We still need to define the values on the edges. Let $S_{i,\sigma}^k = \frac12(S_{i,K}^k+S_{i,K,\sigma}^k)$. To define  $c_{i,\sigma}^k$, we introduce the logarithmic mean for $x,y>0$
\begin{align*}
  L(x,y) &= \int_0^1 x^{1-t}y^t dt
  = \begin{cases}
  \frac{x-y}{\log x - \log y} &\mbox{if } x\neq y, \\
  x &\mbox{if } x = y,
  \end{cases}
\end{align*}
and define
\begin{align*}
    c_{i,\sigma}^k = L(c_{i,K}^k, c_{i,K,\sigma}^k) / \sum_{j=1}^n L(c_{j,K}^k, c_{j,K,\sigma}^k).
\end{align*}
The value $c_{i,\sigma}^k$ approximates $S_i/S$ on $\sigma$ at time $t_k$. By the definition of $L$, it satisfies the following properties:
\begin{align}\label{2.c}
  \sum_{i=1}^n c_{i,\sigma}^k = 1, \quad
  \sum_{i=1}^n c_{i,\sigma}^k(\log c_{i,K,\sigma}^k-\log c_{i,K}^k) = 0.
\end{align}
The quotient
\begin{align}\label{2.quot}
  a_\sigma^k = \frac{\beta({S_{K,\sigma}^k})-\beta({S_{K}^k})
  }{\psi(S_{K,\sigma}^k)-\psi(S_K^k)}
\end{align}
approximates $\beta'(S)/\psi'(S) = a(S)$ on $\sigma$ at time $t_k$. If $S_{K,\sigma}^k=S_K^k$, we set $a_\sigma^k:=a(S_K^k)$. By the generalized mean-value theorem, we have $a_\sigma^k=a(\xi)$ for some $\xi$ between $S_K^k$ and $S_{K,\sigma}^k$. In particular, $\min_{[\kappa,1-\kappa]}a\le a_\sigma^k\le\max_{[\kappa,1-\kappa]}a$ if $S_K^k,S_{K,\sigma}^k\in[\kappa,1-\kappa]$. The definitions of $c_{i,\sigma}^k$ and \eqref{2.quot} are the novel elements of our finite-volume discretization guaranteeing the preservation of the structure.


\section{Discrete energy inequality}\label{sec.dei}

We prove inequality \eqref{1.dei}, recalling definition \eqref{1.Hd} of the discrete total energy $H_d$.

\begin{theorem}\label{thm.dei}
Let Assumptions (A1)--(A4) hold and let $\bm{S}^k\in V_{\Delta x}^n$ be a solution to \eqref{2.eq1}--\eqref{2.eq2} satisfying $S_{i,K}^k>0$, $0<S_K^k<1$. Then
\begin{align*}
  H_d(\bm{S}^k) + c_M\Delta t\sum_{\sigma\in\E}\tau_\sigma
  |\Pi(\D_{K,\sigma}\bm{\mu}^k)|^2
  + \Delta t\sum_{\sigma\in\E}\tau_\sigma
  \D_{K,\sigma}P_c(S^k)\D_{K,\sigma}\beta(S^k)
  \le H_d(\bm{S}^{k-1}).
\end{align*}
\end{theorem}

Observe that the second and third terms on the left-hand side are nonnegative, since $P_c$ and $\beta$ are increasing functions.

\begin{proof}
We multiply \eqref{2.eq1} by $\Delta t\mu_{i,K}^k$ and sum over $K\in\T$ and $i=1,\ldots,n$. This gives $I_1+I_2=0$, where
\begin{align*}
  I_1 = \sum_{i=1}^n\sum_{K\in\T}\m(K)\Phi_K(S_{i,K}^k-S_{i,K}^{k-1})
  \mu_{i,K}^k, \quad
  I_2 = \Delta t\sum_{i=1}^n\sum_{K\in\T}\sum_{\sigma\in\E_K}
  F_{i,K,\sigma}(\bm{S}^k) \mu_{i,K}^k.
\end{align*}
We deduce from the definition of $\bm{\mu}^k$ and the convexity of $E$ that
\begin{align*}
  I_1 &= \sum_{i=1}^n\sum_{K\in\T}\m(K)\Phi_K(S_{i,K}^k-S_{i,K}^{k-1})
  \frac{\pa E}{\pa S_i}(\bm{S}_K^k) \\
  &\ge \sum_{K\in\T}\m(K)\Phi_K
  (E(\bm{S}_K^k)-E(\bm{S}_K^{k-1})).
\end{align*}
We insert definition \eqref{2.eq2} of the numerical flux and integrate by parts:
\begin{align*}
  I_2 &= \Delta t\sum_{i=1}^n\sum_{\sigma\in\E}\tau_\sigma
  c_{i,\sigma}^k a_{\sigma}^k\D_{K,\sigma}\bigg(P_c(S^k)
  + \frac{\beta(S^k)-\beta(S^{k-1})}{\Delta t}\bigg)
  \D_{K,\sigma}\mu_{i}^k \\
  &\phantom{xx}+ \Delta t\sum_{i,j=1}^n\sum_{\sigma\in\E}\tau_\sigma
  M_{ij}(\bm{S}_\sigma^k)
  \D_{K,\sigma}\mu_j^k\D_{K,\sigma}\mu_{i}^k =: I_{21} + I_{22}.
\end{align*}
By Assumption (A2),
\begin{align*}
  I_{22} \ge c_M\Delta t\sum_{\sigma\in\E}\tau_\sigma
  |\Pi(\D_{K,\sigma}\bm\mu^k)|^2.
\end{align*}
We reformulate for $\sigma\in\E$, using the definitions of $\mu_{i,K}^k$ and $c_{i,K}^k$ (recall that $S_i^D$ is constant) as well as property \eqref{2.c},
\begin{align*}
  \sum_{i=1}^n c_{i,\sigma}^k a_\sigma^k\D_{K,\sigma}\mu_i^k
  &= a_\sigma^k\sum_{i=1}^n c_{i,\sigma}^k
  \big(\log c_{i,K,\sigma}^k - \log c_{i,K}^k
  + \psi(S_{K,\sigma}^k) - \psi(S_K^k)\big) \\
  &= a_\sigma^k\big(\psi(S_{K,\sigma}^k) - \psi(S_K^k)\big)
  = \D_{K,\sigma}\beta(S^k),
\end{align*}
where the last identity follows from the definition of $a_\sigma^k$ if $S_{K,\sigma}^k\neq S_K^k$, while both sides vanish if $S_{K,\sigma}^k=S_K^k$. Thus, the term $I_{21}$ becomes
\begin{align*}
  I_{21} &= \Delta t\sum_{\sigma\in\E}\tau_\sigma
  \D_{K,\sigma}\beta(S^k)\,\D_{K,\sigma}\bigg(P_c(S^k)
  + \frac{\beta(S^k)-\beta(S^{k-1})}{\Delta t}\bigg) \\
  &= \sum_{\sigma\in\E}\tau_\sigma
  \D_{K,\sigma}\beta(S^k)\big(\D_{K,\sigma}\beta(S^k)
  - \D_{K,\sigma}\beta(S^{k-1})\big) \\
  &\phantom{xx}+ \Delta t\sum_{\sigma\in\E}\tau_\sigma
  \D_{K,\sigma}\beta(S^k)\D_{K,\sigma}P_c(S^k) \\
  &\ge \frac12\sum_{\sigma\in\E}\tau_\sigma
  \big[\big(\D_{K,\sigma}\beta(S^k)\big)^2
  - \big(\D_{K,\sigma}\beta(S^{k-1})\big)^2\big]
  + \Delta t\sum_{\sigma\in\E}\tau_\sigma
  \D_{K,\sigma}\beta(S^k)\D_{K,\sigma}P_c(S^k).
\end{align*}
Collecting the estimates for $I_1$, $I_{21}$, and $I_{22}$ and taking into account definition \eqref{1.Hd} of $H_d$ concludes the proof.
\end{proof}


\section{Existence of discrete solutions}\label{sec.ex}

The discrete energy inequality in Theorem \ref{thm.dei} provides uniform bounds that allow us to apply a fixed-point argument, leading to the existence of a discrete solution.

\begin{theorem}[Existence of a solution]\label{thm.ex}
Let Assumptions (A1)--(A4) hold and let $\bm{S}^{k-1}\in V_{\Delta x}^n$ with $\bm{S}_K^{k-1}\in\dom$ for all $K\in\T$. Then there exists a solution $\bm{S}^k$ to \eqref{2.eq1}--\eqref{2.eq2} such that $S_i^k>0$ for $i=1,\ldots,n$ and $0<S^k<1$. More precisely, there exists $\kappa\in(0,1/2)$, depending on $\T$, $\Delta t$, and $H_d(\bm{S}^{k-1})$, such that
\begin{align}\label{4.S}
  S_{i,K}^k\ge\kappa>0, \quad  S_K^k\le 1-\kappa
  \quad\mbox{for all }K\in\T,\ i=1,\ldots,n.
\end{align}%
\end{theorem}

\begin{proof}
We first prove the existence of a discrete solution to the regularized problem
\begin{align}\label{2.regul}
 \eps\m(K)\mu_{i,K}^k + \frac{\m(K)}{\Delta t}\Phi_K (S_{i,K}^k
  - S_{i,K}^{k-1}) + \sum_{\sigma\in\E_K}F_{i,K,\sigma}(\bm{S}^k) = 0,
\end{align}
where $\eps>0$, $F_{i,K,\sigma}(\bm{S}^k)$ is given by \eqref{2.eq2}, $S_{i,K}^k = S_i(\bm{\mu}_K^k)$, and $\bm{S}(\bm\mu)$ is the inverse of $\bm{S}\mapsto\bm\mu$. Let $\bm{S}^{k-1}\in V_{\Delta x}^n$ satisfy $\bm{S}^{k-1}\in\dom$. We introduce the fixed-point operator $G:V_{\Delta x}^n\to V_{\Delta x}^n$, $G(\bar{\bm{\mu}}) = \bm{\mu}$, where $\bm{\mu}$ is the unique solution to
\begin{align*}
  \eps\m(K)\mu_{i,K} = - \frac{\m(K)}{\Delta t}\Phi_K
  (S_{i}(\bar{\bm\mu}_K) - S_{i,K}^{k-1})
  - \sum_{\sigma\in\E_K}F_{i,K,\sigma}(\bm{S}(\bar{\bm\mu})),
  \quad K\in\T,\ i=1,\ldots,n.
\end{align*}
By definition of the map $\bm{S}:\R^n\to\dom$, we have $\bm{S}(\bar{\bm\mu}_K)\in\dom$ for all $K\in\T$. The finite-dimensional space $V_{\Delta x}^n$ is endowed with the norm $\|\cdot\|_{0,2,\T}$ defined in Section \ref{sec.num}. The operator $G$ is well defined, and standard arguments show that it is continuous. We want to apply Schaefer's fixed-point theorem \cite[Sec.~9.2.2]{Eva98}. For this, let $\theta\in[0,1]$ and $\bm{\mu} = \theta G(\bm\mu)$, i.e., $\mu_i$ solves for $K\in\T$, $i=1,\ldots,n$,
\begin{align*}
  \eps\m(K)\mu_{i,K} = -\theta\frac{\m(K)}{\Delta t}\Phi_K
  (S_{i}(\bm\mu_K) - S_{i,K}^{k-1})  - \theta\sum_{\sigma\in\E_K}F_{i,K,\sigma}(\bm{S}(\bm\mu)).
\end{align*}
Proceeding as in Section \ref{sec.dei}, we can prove that
\begin{align}\label{2.dei}
  \theta H_d(\bm{S}(\bm\mu)) &+ \eps\Delta t\|\bm\mu\|_{0,2,\T}^2
  + \theta c_M\Delta t\sum_{\sigma\in\E}\tau_\sigma
  |\Pi(\D_{K,\sigma}\bm{\mu})|^2 \\
  &+ \theta\Delta t\sum_{\sigma\in\E}\tau_\sigma
  \D_{K,\sigma}P_c(S)\D_{K,\sigma}\beta(S) \le \theta H_d(\bm{S}^{k-1}).
  \nonumber
\end{align}
Since $E_{\rm bulk}(\bm{S})=\sum_{i=1}^n S_i\log((S_i/S)/(S_i^D/S^D))\ge 0$ (by Jensen's inequality) and $\Psi\ge 0$, we have $H_d\ge 0$. Therefore, \eqref{2.dei} shows that $\bm\mu$ is bounded uniformly in $\theta\in[0,1]$ (for fixed $\eps>0$). By Schaefer's fixed-point theorem, there exists a solution $\bm\mu^\eps:=\bm\mu$ to \eqref{2.regul}. We set $\bm{S}^\eps:=\bm{S}(\bm\mu^\eps)$. It holds that $\bm{S}^\eps_K\in\dom$ for all $K\in\T$.

It remains to perform the limit $\eps\to 0$. The energy inequality and the explicit form of $\Psi$ yield the following bound that is uniform in $\eps$:
\begin{align*}
  \sum_{K\in\T}\m(K)\big((S_K^\eps)^{2-\gamma_2}
  + (1-S_K^\eps)^{2-\gamma_1}\big) \le C.
\end{align*}
Since $\gamma_1$, $\gamma_2>2$, there exists $\kappa\in(0,1)$ such that $\kappa\le S_K^\eps\le 1-\kappa$ for all $K\in\T$, uniformly in $\eps$. These bounds show that $\|S_i^\eps\|_{0,\infty,\T}$ is bounded uniformly in $\eps$. Thus, there exists a subsequence (not relabeled) such that
\begin{align*}
  S_i^\eps\to S_i \quad\mbox{as }\eps\to 0.
\end{align*}

Furthermore, we deduce from the energy inequality that $(\Pi\bm\mu^\eps)$ is bounded in the discrete $H^1(\Omega)$ seminorm $|\cdot|_{1,2,\mathcal{T}}$. Since $\mu_i^D=0$ is given on $\Gamma_D$, we can apply the discrete Poincar\'e inequality to infer that $\Pi\bm\mu^\eps$ is uniformly bounded in $L^2(\Omega)$. Hence, up to a subsequence,
\begin{align}
  \Pi\bm\mu^\eps\to \bm{Q}\quad\mbox{as }\eps\to 0.
  \label{conv_Pi.mu}
\end{align}
Since $\Pi\bm\mu^\eps\in\operatorname{ran}(\Pi)$ for all $\eps>0$
and $\operatorname{ran}(\Pi)=\mbox{span}\{\mathbf 1\}^{\perp}$ is a closed subspace of $\R^n$, the convergence \eqref{conv_Pi.mu}
implies that $\bm Q\in\operatorname{ran}(\Pi)$. Consequently, there exists $\bm\mu\in\R^n$ such that $\bm Q=\Pi\bm\mu$. To identify $\bm\mu$, we proceed as in \cite{DMZ25}. We know that
\begin{align*}
  \D_{K,\sigma}&(\log S_i^\eps-\log S_i^D)
  - \frac{1}{n}\sum_{j=1}^n\D_{K,\sigma}(\log S_j^\eps-\log S_j^D) \\
  &= \D_{K,\sigma}\bigg(\log\frac{S_i^\eps}{S^\eps}
  - \log\frac{S_i^D}{S^D} + \psi(S^\eps)\bigg)
  - \frac{1}{n}\sum_{j=1}^n\D_{K,\sigma}
  \bigg(\log\frac{S_j^\eps}{S^\eps} - \log\frac{S_j^D}{S^D}
  + \psi(S^\eps)\bigg) \\
  &= \D_{K,\sigma}\mu_i^\eps
  - \frac{1}{n}\sum_{j=1}^n\D_{K,\sigma}\mu_j^\eps
  = (\Pi(\D_{K,\sigma}\bm\mu^\eps))_i
\end{align*}
and consequently $\log S_i^\eps-\log S_i^D-(1/n)\sum_{j=1}^n(\log S_j^\eps-\log S_j^D)$ is uniformly bounded in $L^2(\Omega)$. We apply \cite[Lemma 17]{DMZ25} to $c_{i,K}^\eps= S_{i,K}^\eps/S_K^\eps$, which satisfy $\sum_{i=1}^n c_{i,K}^\eps=1$: For any $\delta>0$, there exists $C_\delta>0$ such that
\begin{align*}
  \sum_{i=1}^n\bigg(\log c_{i,K}^\eps - \frac{1}{n}\sum_{j=1}^n
  \log c_{j,K}^\eps\bigg)^2
  \ge C_\delta\bigg(\sum_{i=1}^n\log c_{i,K}^\eps\bigg)^2 - \delta.
\end{align*}
Since $\log c_{i,K}^\eps-\frac1n\sum_{j=1}^n\log c_{j,K}^\eps=\log S_{i,K}^\eps-\frac1n\sum_{j=1}^n\log S_{j,K}^\eps$ and the boundary data are constant, the left-hand side is bounded uniformly in $\eps$. Hence, $|\sum_{i=1}^n\log c_{i,K}^\eps|\le C$, and since $\log c_{i,K}^\eps\le 0$ for all $i$, we infer that $|\log c_{i,K}^\eps|\le C$. Together with $\kappa\le S_K^\eps\le 1$, this shows that $\log S_{i,K}^\eps=\log c_{i,K}^\eps+\log S_K^\eps$ is uniformly bounded for all $K\in\T$. We obtain, up to a subsequence,
\begin{align*}
  \log S_i^\eps\to\log S_i\quad\mbox{as }\eps\to 0.
\end{align*}
We conclude that $S_i>0$ in $\Omega$. Thus, we can identify
\begin{align*}
  \mu_i = \log\frac{S_i}{S} - \log\frac{S_i^D}{S^D} + \psi(S)
  \quad\mbox{for }i=1,\ldots,n.
\end{align*}
Inequality \eqref{2.dei} provides a uniform bound for $\sqrt\eps\mu_{i}^\eps$, which shows (up to a subsequence) that $\eps\mu_{i,K}\to 0$. Hence, we can perform the limit $\eps\to 0$ in \eqref{2.regul} with numerical fluxes \eqref{2.eq2} showing that $\bm{S}^k:=\bm{S}$ solves \eqref{2.eq1}--\eqref{2.eq2}.
\end{proof}


\section{Uniqueness of a discrete solution}\label{sec.unique}

In this section, we prove the uniqueness of a solution to our scheme.

\begin{theorem}\label{thm.unique}
Let Assumptions (A1)--(A4) hold, let $\kappa\in(0,1/2)$, and let $\bm{S}^{k-1}\in V_{\Delta x}^n$ with $\bm{S}_K^{k-1}\in\dom$ for all $K\in\T$. Then there exists $\Delta t_0>0$, depending on $\kappa$, the mesh $\T$, and the data, such that for all $0<\Delta t\le\Delta t_0$, there exists at most one solution $\bm{S}^k$ to \eqref{2.eq1}--\eqref{2.eq2} satisfying the bounds \eqref{4.S}.
\end{theorem}

\begin{proof}
Let $\bm{S}^k$ and $\bar{\bm{S}}^k$ be two discrete solutions at time $t_k$ with the respective potentials $\bm{\mu}^k$ and $\bar{\bm{\mu}}^k$. We assume that they share the same previous state, i.e.\ $\bm{S}^{k-1}=\bar{\bm{S}}^{k-1}$. We notice that, since we consider a fully discrete, finite-dimensional setting on a fixed mesh $\mathcal{T}$, all discrete spatial gradients are bounded. Moreover, by assumption, both solutions satisfy the bounds \eqref{4.S}. Since $S_K^k\ge S_{i,K}^k$, we have $S_K^k,\bar S_K^k\in[\kappa,1-\kappa]$ for all $K\in\T$. Without loss of generality, we may assume that $\kappa\le\min_i S_i^D$ and $S^D\le 1-\kappa$ (otherwise, we replace $\kappa$ by a smaller value), such that the Dirichlet boundary values satisfy the same bounds. In the following, the constants $C,C_1,C_*,\ldots>0$ depend on $\kappa$, the mesh $\T$, and the data, but not on $\Delta t$.

	In the first step, we show the uniqueness for the water saturations, i.e., we prove that ${S}^k = \bar{S}^k$.
	 By summing equations \eqref{2.eq1} and using the identities $\sum_i M_{ij} = 0$, $\sum_{i} c_{i,\sigma}^k = 1$, we obtain
	\begin{align}
		& \frac{\m(K)}{\Delta t}\Phi_K(S_{K}^k - S_{K}^{k-1})
		+ \sum_{\sigma\in\E_K}F_{K,\sigma}({S}^k) = 0,
		\quad K\in\T,\ k\in\N, \label{2.eq1-sum} \\
		& F_{K,\sigma}({S}^k) = -\tau_\sigma  a_{\sigma}^k\D_{K,\sigma}\bigg(P_c(S^k)
		+ \frac{\beta(S^k)-\beta(S^{k-1})}{\Delta t}\bigg) . 	\label{2.eq2-sum}
	\end{align}
	With the notation $w^k = \beta(S^k)-\beta(S^{k-1})$, these equations can be written as
\begin{align*}
\m(K)\Phi_K(S_{K}^k - S_{K}^{k-1})
- \sum_{\sigma\in\E_K}\tau_\sigma  a_{\sigma}^k \bigg(\Delta t\D_{K,\sigma} P_c(S^k) + \D_{K,\sigma} w^k \bigg) = 0.
\end{align*}
Let $K\in\T$ be a cell in which $w^k$ attains its maximum, and assume that this maximum is positive. Then $\D_{K,\sigma} w^k \leq 0$ (note that $w^k=0$ on Dirichlet edges) and consequently,
\begin{align*}
	\m(K)\Phi_K(S_{K}^k - S_{K}^{k-1})
	\leq \Delta t \sum_{\sigma\in\E_K}\tau_\sigma  a_{\sigma}^k\D_{K,\sigma}P_c(S^k)  .
\end{align*}
We deduce from the boundedness of the functions $a_{\sigma}^k$ and $P_c(S^k)$ on the saturation range $[\kappa,1-\kappa]$ that
\begin{align*}
	S_{K}^k - S_{K}^{k-1}
	\leq \frac{C\Delta t}{\Phi_K \m(K)} \sum_{\sigma\in\E_K} \tau_\sigma  \le \Delta t C R_{\T},
\end{align*}
where $R_{\T} = \max_{K\in \T} \sum_{\sigma\in\E_K} \tau_\sigma/\m(K)$. Since $\beta$ is increasing and Lipschitz continuous (see Lemma \ref{lem.beta} below), we conclude that $	w_{K}^k \leq\Delta t C R_{\T}$. At the negative minimum, the reversed inequality gives the corresponding lower bound. We infer that
 \begin{align}\label{bound-w1}
 	\max_{K\in \T} |w_{K}^k| \leq  \Delta t C R_{\T}.
 \end{align}
We deduce from this bound and the boundedness of the capillary pressure that
 \begin{align}
 	\label{bound-w2}
 	\max_{K\in \T} | \D_{K,\sigma}(\Delta t P_c(S^k) +  w^k )  |  \leq  \Delta t C (1+R_{\T}).
 \end{align}

Let $(S_{K}^k)_{K\in\T}$ and $(\bar{S}_{K}^k)_{K\in\T}$ be two solutions to \eqref{2.eq1-sum}--\eqref{2.eq2-sum}. We subtract these equations, multiply the difference by $\Delta t(\beta(S_K^k)-\beta(\bar{S}_K^{k}))$, and sum over all elements. This gives $I_1+I_2=0$, where
\begin{align*}
	I_1 &= \sum_{K\in\mathcal{T}}\m(K)\Phi_K
	(S_{K}^k-\bar{S}_{K}^k)
    \big(\beta(S^k_K)-\beta(\bar{S}_{K}^{k})\big), \\
	I_2 &= \Delta t\sum_{K\in\mathcal{T}}\sum_{\sigma\in\mathcal{E}_K}
	\big(F_{K,\sigma}({S}^k) - F_{K,\sigma}(\bar{S}^k)\big)
	\big(\beta(S^k_K)-\beta(\bar{S}_{K}^{k})\big).
\end{align*}
Since the function $\beta$ is increasing and Lipschitz continuous, we have, with $c_\beta=\|\beta'\|^{-1}_{L^\infty(0,1)}$,
\begin{align*}
	I_1 \ge c_\beta c_\Phi\sum_{K\in\mathcal{T}}
	\m(K)|\beta({S}_K^k)-\beta(\bar{{S}}_K^k)|^2
	= C\|\beta({S}^k)- \beta(\bar{{S}}^k)\|_{0,2,\mathcal{T}}^2.
\end{align*}
We apply discrete integration by parts in $I_2$ and use definition \eqref{2.eq2-sum} as well as $w^k = \beta(S^k)-\beta(S^{k-1})$, $ \bar{w}^k =\beta(\bar{S}^k)-\beta(\bar{S}^{k-1})$:
\begin{align*}
	I_2 &= -\Delta t\sum_{\sigma\in\mathcal{E}}
	(F_{K,\sigma}({S}^k) - F_{K,\sigma}(\bar{S}^k) ) \D_{K,\sigma}(\beta(S^k)-\beta(\bar{S}^{k})) \\
	&= \sum_{\sigma\in\mathcal{E}}\tau_\sigma
	\big\{a_{\sigma}^k\D_{K,\sigma}\big(\Delta t P_c(S^k) + w^k\big)
	 - \bar{a}_{\sigma}^k\D_{K,\sigma}\big( \Delta t P_c(\bar{S}^k)
	 + \bar{w}^k\big)\big\} \\
    &\phantom{xx}\times\D_{K,\sigma}(\beta(S^k)-\beta(\bar{S}^{k})).
\end{align*}
Since the previous state is the same for both solutions, we have $u^k := w^k - \bar{w}^k = \beta(S^k)-\beta(\bar{S}^{k})$. We write $I_2 = I_{21} + I_{22}$, where
\begin{align*}
	I_{21} &= \sum_{\sigma\in\mathcal{E}}\tau_\sigma
	 a_{\sigma}^k \D_{K,\sigma}\big\{\Delta t (P_c(S^k) - P_c(\bar{S}^k))	+ u^k\big\}\D_{K,\sigma} u^{k}\\
	 & = \sum_{\sigma\in\mathcal{E}}\tau_\sigma a_{\sigma}^k (\D_{K,\sigma} u^k)^2 + \Delta t \sum_{\sigma\in\mathcal{E}}\tau_\sigma
	 a_{\sigma}^k \D_{K,\sigma}(P_c(S^k) - P_c(\bar{S}^k))
     \D_{K,\sigma} u^{k}, \\
	 I_{22} & = \sum_{\sigma\in\mathcal{E}}\tau_\sigma
	 (a_{\sigma}^k -\bar{a}_{\sigma}^k)\D_{K,\sigma}
     \big( \Delta t P_c(\bar{S}^k)
	 + \bar{w}^k\big) \D_{K,\sigma} u^{k}.
\end{align*}
It follows from $I_1+I_2 = 0$ that
\begin{align*}
C\| u^k  \|_{0,2,\mathcal{T}}^2 +\sum_{\sigma\in\mathcal{E}}\tau_\sigma a_{\sigma}^k (\D_{K,\sigma} u^k)^2
\leq  & -\Delta t \sum_{\sigma\in\mathcal{E}}\tau_\sigma
a_{\sigma}^k \D_{K,\sigma}(P_c(S^k) - P_c(\bar{S}^k))  \D_{K,\sigma} u^{k} \\
& - \sum_{\sigma\in\mathcal{E}}\tau_\sigma
(a_{\sigma}^k -\bar{a}_{\sigma}^k)\D_{K,\sigma}
\big( \Delta t P_c(\bar{S}^k)
+ \bar{w}^k\big)  \D_{K,\sigma} u^{k}.
\end{align*}
Using the Lipschitz continuity of $P_c$ on $[\kappa,1-\kappa]$ and of $\beta^{-1}$ on $\beta([\kappa,1-\kappa])$ (with Lipschitz constant $\beta_*:=1/\min_{[\kappa,1-\kappa]}\beta'$), we can estimate
\begin{align*}
	|\D_{K,\sigma}(P_c(S^k) - P_c(\bar{S}^k)) | \leq C(|S_K^k - \bar{S}_K^k| + |S_{K,\sigma}^k - \bar{S}_{K,\sigma}^k|)
	\leq C\beta_* (|u_K^k| + |u_{K,\sigma}^k|),
\end{align*}
and also $|a_{\sigma}^k -\bar{a}_{\sigma}^k | \leq C\beta_* (|u_K^k| + |u_{K,\sigma}^k|)$, since the map $(S,S')\mapsto(\beta(S')-\beta(S))/(\psi(S')-\psi(S))$ is Lipschitz continuous on $[\kappa,1-\kappa]^2$. We deduce from the boundedness of $a_{\sigma}^k$, $a_{\sigma}^k\geq a_* >0$ (with $a_*:=\min_{[\kappa,1-\kappa]}a$; see the remark after \eqref{2.quot}), and bound \eqref{bound-w2} that
\begin{align*}
	C\| u^k  \|_{0,2,\mathcal{T}}^2
    &+ a_*\sum_{\sigma\in\mathcal{E}}\tau_\sigma  (\D_{K,\sigma} u^k)^2
	\leq  \Delta t C (1+R_{\T})  \sum_{\sigma\in\mathcal{E}}\tau_\sigma | \D_{K,\sigma} u^{k}|(|u_K^k| + |u_{K,\sigma}^k|)\\
	&\leq \frac{a_*}{2}\sum_{\sigma\in\mathcal{E}}\tau_\sigma  (\D_{K,\sigma} u^k)^2
	+   C_* (\Delta t)^2(1+R_{\T})^2   \sum_{\sigma\in\mathcal{E}}\tau_\sigma (|u_K^k| + |u_{K,\sigma}^k|)^2,
\end{align*}
for some  $C_*>0$ depending on $a_*$. This shows, using $\sum_{\sigma\in\E}\tau_\sigma(|u_K^k|+|u_{K,\sigma}^k|)^2\le 4R_\T\|u^k\|_{0,2,\T}^2$, that
\begin{align*}
	C\| u^k  \|_{0,2,\mathcal{T}}^2 + \frac{a_*}{2}\sum_{\sigma\in\mathcal{E}}\tau_\sigma  (\D_{K,\sigma} u^k)^2
	\leq
	  C_1 (\Delta t)^2(1+R_{\T})^{3}   \|u^k\|_{0,2,\mathcal{T}}^2 .
\end{align*}
For sufficiently small $\Delta t>0$ (namely, if $C_1(\Delta t)^2(1+R_\T)^3<C$), we arrive at $\|u^k\|_{0,2,\T}^2 \le 0$ and consequently $S^k = \bar{S}^k$. This completes the proof of
uniqueness of the phase saturation.

In the second step, we pass to the proof of uniqueness of the components $S_i^k$. We take the difference of equations \eqref{2.eq1}, satisfied for $S_{i,K}^k$ and $\bar{S}_{i,K}^k$, multiply it by $\Delta t (\mu_{i,K}^k - \bar{\mu}_{i,K}^k)$, and sum it over $K\in\mathcal{T}$ and $i=1,\ldots,n$, leading to $I_3+I_4=0$, where
\begin{align*}
  I_3 &= \sum_{i=1}^n\sum_{K\in\mathcal{T}}\m(K)\Phi_K
  (S_{i,K}^k-\bar{S}_{i,K}^k)(\mu_{i,K}^k-\bar\mu_{i,K}^{k}), \\
  I_4 &= \Delta t\sum_{i=1}^n\sum_{K\in\mathcal{T}}\sum_{\sigma\in\mathcal{E}_K}
  \big(F_{i,K,\sigma}(\bm{S}^k) - F_{i,K,\sigma}(\bar{\bm{S}}^k)\big)
  (\mu_{i,K}^k-\bar\mu_{i,K}^{k}).
\end{align*}
Starting with $I_3$, we observe that $\mu_{i,K}^k = (\pa E/\pa S_i)(\bm{S}_K^k)$ and hence, the sum over the components can be written in terms of the gradient of the free energy,
\begin{align*}
  \sum_{i=1}^n(S_{i,K}^k-\bar{S}_{i,K}^k)(\mu_{i,K}^k-\bar\mu_{i,K}^{k})
  = (\bm{S}_K^k-\bar{\bm{S}}_K^k)\cdot\big(\na E(\bm{S}_K^k)
  - \na E(\bar{\bm{S}}_K^k)\big).
\end{align*}
Since $0<\kappa\le S_{i,K}^k$, $\bar{S}_{i,K}^k\le 1-\kappa<1$, for all $i=1,\ldots,n$,
the Hessian $\D^2E(\bm{S})=\mathrm{diag}(1/S_1,\ldots,1/S_n)-(1/S-\psi'(S))\,\mathrm{1}\otimes\mathrm{1}$ is positive definite (by the Cauchy--Schwarz inequality), uniformly on the compact convex set $\{\bm{S}\in\R^n:S_i\ge\kappa,\ S\le 1-\kappa\}$. Hence, $E$ is strongly convex on this set. This yields a constant $c_E>0$ such that
\begin{align*}
  (\bm{S}_K^k-\bar{\bm{S}}_K^k)\cdot\big(\na E(\bm{S}_K^k)
  - \na E(\bar{\bm{S}}_K^k)\big)
  \ge c_E|\bm{S}_K^k-\bar{\bm{S}}_K^k|^2.
\end{align*}
Then, together with the positive lower bound for $\Phi_K$, we obtain
\begin{align*}
  I_3 \ge c_Ec_\Phi\sum_{K\in\mathcal{T}}
  \m(K)|\bm{S}_K^k-\bar{\bm{S}}_K^k|^2
  = c_Ec_\Phi\|\bm{S}^k-\bar{\bm{S}}^k\|_{0,2,\mathcal{T}}^2.
\end{align*}

We turn to the second term $I_4$. By discrete integration by parts \eqref{2.dibp},
\begin{align*}
  I_4 = -\Delta t\sum_{i=1}^n\sum_{\sigma\in\mathcal{E}}
  \big(F_{i,K,\sigma}(\bm{S}^k) - F_{i,K,\sigma}(\bar{\bm{S}}^k)\big)
  \D_{K,\sigma}(\mu_i^k-\bar{\mu}_i^k).
\end{align*}
We split $I_4 = I_{41}+I_{42}$, where (using the fact that $S^k = \bar{S}^k$ and $a_\sigma^k = \bar{a}_\sigma^k$)
\begin{align*}
  I_{41} &= \sum_{\sigma\in\mathcal{E}}\tau_\sigma
  a_\sigma^k\D_{K,\sigma}
  \big(\Delta t P_c(S^k) + \beta(S^k)-\beta(S^{k-1})\big) \sum_{i=1}^n (c_{i,\sigma}^k - \bar{c}_{i,\sigma}^k)
  \D_{K,\sigma}(\mu_i^k-\bar\mu_i^{k}), \\
  I_{42} &= \Delta t\sum_{i=1}^n\sum_{\sigma\in\mathcal{E}}\tau_\sigma
  \sum_{j=1}^n\bigg(M_{ij}(\bm{S}_\sigma^k)\D_{K,\sigma}\mu_j^k
  - M_{ij}(\bar{\bm{S}}_\sigma^k)\D_{K,\sigma}\bar\mu_j^k\bigg)
  \D_{K,\sigma}(\mu_i^k-\bar\mu_i^{k}),
\end{align*}
where $\bar{c}_{i,\sigma}^k$  are evaluated at $\bar{\bm{S}}^k$. We write $I_{42} = I_{421} + I_{422}$, where
\begin{align*}
  I_{421} &= \Delta t\sum_{\sigma\in\mathcal{E}}\tau_\sigma\sum_{i,j=1}^n
  M_{ij}(\bm{S}_\sigma^k)\D_{K,\sigma}(\mu_j^k-\bar\mu_j^{k})
  \D_{K,\sigma}(\mu_i^k-\bar\mu_i^{k}), \\
  I_{422} &= \Delta t\sum_{\sigma\in\mathcal{E}}\tau_\sigma\sum_{i,j=1}^n
  \big(M_{ij}(\bm{S}_\sigma^k) - M_{ij}(\bar{\bm{S}}_\sigma^k)\big)
  \D_{K,\sigma}\bar\mu_j^k\D_{K,\sigma}(\mu_i^k-\bar\mu_i^{k}).
\end{align*}
By Assumption (A2),
\begin{align*}
  I_{421} \ge c_M\Delta t\sum_{\sigma\in\mathcal{E}}\tau_\sigma
  |\Pi\D_{K,\sigma}(\bm\mu^k-\bar{\bm{\mu}}^k)|^2.
\end{align*}
Taking into account $\sum_{i=1}^n M_{ij}=0$ 
we have
\begin{align*}
  \sum_{i=1}^n\big(M_{ij}(\bm{S}_\sigma^k)
  - M_{ij}(\bar{\bm{S}}_\sigma^k)\big) = 0.
\end{align*}
This allows us to replace the plain difference with its projection:
\begin{align*}
  I_{422} = \Delta t\sum_{\sigma\in\mathcal{E}}\tau_\sigma\sum_{i,j=1}^n
  \big(M_{ij}(\bm{S}_\sigma^k) - M_{ij}(\bar{\bm{S}}_\sigma^k)\big)
  \D_{K,\sigma}\bar\mu_j^k
  \big(\Pi\D_{K,\sigma}(\bm\mu^k-\bar{\bm{\mu}}^{k})\big)_i.
\end{align*}
Then, using the local Lipschitz continuity of $M_{ij}$ and Young's inequality,
\begin{align*}
  I_{422} \ge -\frac{c_M}{2}\Delta t\sum_{\sigma\in\mathcal{E}}
  \tau_\sigma|\Pi\D_{K,\sigma}(\bm\mu^k-\bar{\bm{\mu}}^{k})|^2
  - C(c_M)\Delta t\sum_{\sigma\in\mathcal{E}}\tau_\sigma
  |\D_{K,\sigma}\bar{\bm{\mu}}^k|^2
  |\bm{S}_\sigma^k-\bar{\bm{S}}_\sigma^k|^2.
\end{align*}
The first term on the right-hand side is absorbed by $I_{421}$. Thanks to the bounds $\kappa\le S_{i,K}^k$, $\bar S_{i,K}^k\le 1-\kappa$, for $i=1,\ldots,n$, the functions defining the chemical potentials stay strictly away from the singularities.
Therefore, $\bar{\bm{\mu}}^k$ is bounded by a constant depending on $\kappa$, and we infer that
\begin{align*}
  \max_{\sigma\in\mathcal{E}}|\D_{K,\sigma}\bar{\bm{\mu}}^k|^2
  \le C(\kappa),
\end{align*}
for some $C(\kappa)>0$. This shows that
\begin{align*}
  I_{42} \ge \frac{c_M}{2}\Delta t\sum_{\sigma\in\mathcal{E}}
  \tau_\sigma|\Pi\D_{K,\sigma}(\bm\mu^k-\bar{\bm{\mu}}^{k})|^2
  - C(c_M,\kappa)\Delta t\sum_{\sigma\in\mathcal{E}}\tau_\sigma
  |\bm{S}_\sigma^k-\bar{\bm{S}}_\sigma^k|^2.
\end{align*}

Next, we estimate $I_{41}$, noting first that this term can be rewritten as
\begin{align*}
	I_{41} &=
		\sum_{\sigma\in\mathcal{E}}\tau_\sigma 	a_\sigma^k\D_{K,\sigma}
		\big\{\Delta t P_c(S^k) + \beta(S^k)-\beta(S^{k-1})\big\} (\bm{c}_{\sigma}^k - \bar{\bm{c}}_{\sigma}^k ) \cdot \Pi \D_{K,\sigma}( \bm{\mu}^k - \bar{\bm{\mu}}^{k} ).
\end{align*}
Using \eqref{bound-w2} and the boundedness of $a_\sigma^k$, we can estimate
\begin{align*}
	|I_{41}| &\leq \Delta t C (1+R_{\T})
		\sum_{\sigma\in\mathcal{E}}\tau_\sigma  |\bm{c}_{\sigma}^k - \bar{\bm{c}}_{\sigma}^k | |\Pi \D_{K,\sigma}( \bm{\mu}^k - \bar{\bm{\mu}}^{k})| \\
		& \leq \frac{c_M}{4}\Delta t\sum_{\sigma\in\mathcal{E}}
		\tau_\sigma|\Pi\D_{K,\sigma}(\bm\mu^k-\bar{\bm{\mu}}^{k})|^2
		+ \Delta t C(c_M) (1+R_{\T})^{2}	\sum_{\sigma\in\mathcal{E}}\tau_\sigma  |\bm{c}_{\sigma}^k - \bar{\bm{c}}_{\sigma}^k |^2.
\end{align*}
We collect all previous estimates:
\begin{align*}
	c_Ec_\Phi\|&\bm{S}^k-\bar{\bm{S}}^k\|_{0,2,\mathcal{T}}^2
    + \frac{c_M}{4}\Delta t\sum_{\sigma\in\mathcal{E}}
	 \tau_\sigma|\Pi\D_{K,\sigma}(\bm\mu^k-\bar{\bm{\mu}}^{k})|^2 \\
	 &\leq  C(c_M,\kappa)\Delta t\sum_{\sigma\in\mathcal{E}}\tau_\sigma
	 |\bm{S}_\sigma^k-\bar{\bm{S}}_\sigma^k|^2
     + \Delta t C(c_M) (1+R_{\T})^{2}	\sum_{\sigma\in\mathcal{E}}\tau_\sigma  |\bm{c}_{\sigma}^k - \bar{\bm{c}}_{\sigma}^k |^2.
\end{align*}

It remains to estimate the first term on the right-hand side. Since the mesh is fixed and $\bm{S}_{\sigma}^k = \frac12(\bm{S}_K^k+\bm{S}_{K,\sigma}^k)$, we find that
\begin{align*}
  \sum_{\sigma\in\mathcal{E}}\tau_\sigma
  |\bm{S}_\sigma^k-\bar{\bm{S}}_\sigma^k|^2
  &\le  R_\T   \sum_{K\in\mathcal{T}}\m(K)|\bm{S}_K^k-\bar{\bm{S}}_K^k|^2  = R_\T   \|\bm{S}^k-\bar{\bm{S}}^k\|_{0,2,\mathcal{T}}^2.
\end{align*}
Similarly, since the logarithmic mean is a Lipschitz continuous function on any compact subset of $(0,\infty)\times (0,\infty)$ and the denominator in the definition of $c_{i,\sigma}^k$ satisfies $\sum_{j=1}^n L(c_{j,K}^k,c_{j,K,\sigma}^k)\ge\sum_{j=1}^n\min\{c_{j,K}^k,c_{j,K,\sigma}^k\}\ge n\kappa$ (recall that $c_{j,K}^k=S_{j,K}^k/S_K^k\ge\kappa$ and $c_j^D\ge\kappa$ by our choice of $\kappa$),
we can estimate on any $\sigma\in \E$,
\begin{align*}
	|\bm{c}_{\sigma}^k - \bar{\bm{c}}_{\sigma}^k |^2
    \leq M_1 \big(|\bm{c}_{K,\sigma}^k - \bar{\bm{c}}_{K,\sigma}^k |^2 + |\bm{c}_{K}^k - \bar{\bm{c}}_{K}^k |^2\big)
	\leq M_2 \big(|\bm{S}_{K,\sigma}^k - \bar{\bm{S}}_{K,\sigma}^k |^2 + |\bm{S}_{K}^k - \bar{\bm{S}}_{K}^k |^2\big),
\end{align*}
where we have used the identities $S_{K,\sigma} = \bar{S}_{K,\sigma} $ and $S_{K} = \bar{S}_{K} $. Using again the fact that our mesh is fixed, we infer that
\begin{align*}
	c_Ec_\Phi\|\bm{S}^k-\bar{\bm{S}}^k\|_{0,2,\mathcal{T}}^2 + \frac{c_M}{4}\Delta t\sum_{\sigma\in\mathcal{E}}
	\tau_\sigma|\Pi\D_{K,\sigma}(\bm\mu^k-\bar{\bm{\mu}}^{k})|^2
	\leq C_2\Delta t \|\bm{S}^k-\bar{\bm{S}}^k\|_{0,2,\mathcal{T}}^2.
\end{align*}
Finally, we choose $\Delta t_0>0$ such that $C_2\Delta t_0<c_Ec_\Phi$ and such that the smallness condition of the first step holds for $\Delta t\le\Delta t_0$. By construction, $\Delta t_0$ depends only on $\kappa$, $\T$, and the data. For $0<\Delta t\le\Delta t_0$, we find that $\|\bm{S}^k-\bar{\bm{S}}^k\|_{0,2,\mathcal{T}}=0$ and hence $\bm{S}^k=\bar{\bm{S}}^k$. This completes the proof.
\end{proof}

\begin{remark}[Smallness of $\Delta t$]\label{rem.mesh}\rm
We discuss the dependence of the smallness condition on $\Delta t$ in Theorem \ref{thm.unique} on the discretization parameters. Let $\kappa_0\in(0,1/2)$ and $\kappa_1\in(0,\kappa_0]$ be such that $\kappa_0\le S_K^k,\bar S_K^k\le 1-\kappa_0$ and $S_{i,K}^k,\bar S_{i,K}^k\ge\kappa_1$ for all $K\in\T$ and $i=1,\ldots,n$. The first step of the proof (uniqueness of the total saturation $S^k$) shows that $S^k$ is unique if $\Delta t\le c_1(1+R_\T)^{-3/2}$, where $c_1>0$ depends on $\kappa_0$. For quasi-uniform meshes, we have $R_\T\le C(\Delta x)^{-2}$, and this condition becomes $\Delta t\le c_1(\Delta x)^3$. The second step (uniqueness of the components $S_i^k$) requires that $C_2\Delta t(1+R_\T)^3<c_Ec_\Phi$, i.e.\ $\Delta t\le c_2(\Delta x)^6$ for quasi-uniform meshes. Here, $C_2$ depends on $\kappa_1$ through the bound for $|\D_{K,\sigma}\bar{\bm{\mu}}^k|$ and the Lipschitz constant of the logarithmic mean, which may grow like $|\log\kappa_1|^2$ and $\kappa_1^{-1}$, respectively, as $\kappa_1\to 0$. The lower bound $\kappa_1$ obtained in the proof of Theorem \ref{thm.ex} is derived from the estimate $c_M\Delta t|\Pi\bm\mu^k|_{1,2,\T}^2\le H_d(\bm{S}^{k-1})$ and hence deteriorates as $\Delta t\to 0$. For this reason, Theorem \ref{thm.unique} is formulated as a uniqueness result in the class of solutions satisfying \eqref{4.S} for a given $\kappa$ (i.e.\ $\kappa_0=\kappa_1=\kappa$), and $\Delta t_0$ depends on $\kappa$ and the mesh.
\end{remark}


\section{Stability estimates}\label{sec.stab}

We derive estimates uniform in $\eta=\max\{\Delta x,\Delta t\}$. To this end, we sum the discrete energy inequality in Theorem \ref{thm.dei} over $k=1,\ldots,j$ for $j\le N$:
\begin{align*}
  H_d(\bm{S}^j) + c_M\Delta t\sum_{k=1}^j\sum_{\sigma\in\E}\tau_\sigma
  |\Pi(\D_{K,\sigma}\bm\mu^k)|^2
  + \Delta t\sum_{k=1}^j\sum_{\sigma\in\E}
  \tau_\sigma\D_{K,\sigma}P_c(S^k)
  \D_{K,\sigma}\beta(S^k) \le H_d(\bm{S}^0).
\end{align*}
Since $H_d\ge 0$ and $H_d(\bm{S}^0)$ is bounded uniformly in $\eta$ (this follows from Assumption (A1), since $E$ from \eqref{1.E} is bounded on $\{\bm{S}:S_i\ge c,\ S\le 1-c\}$ and $|\beta(S^0)|_{1,2,\T}\le C\|\na\beta(S^0)\|_{L^2(\Omega)}$ under the usual mesh regularity condition), we obtain the following bounds uniformly in $\eta$:
\begin{equation}\label{5.est1}
\begin{aligned}
  \sup_{k=1,\ldots,N}\big(\|(S^k)^{2-\gamma_2}\|_{0,1,\T}
  + \|(1-S^k)^{2-\gamma_1}\|_{0,1,\T}\big) &\le C, \\
  \Delta t\sum_{k=1}^N|\Pi(\bm\mu^k)|_{1,2,\T}^2
  + \sup_{k=1,\ldots,N}|\beta(S^k)|_{1,2,\T}^2 &\le C.
\end{aligned}
\end{equation}

We first prove a result that is essential for the additional stability estimate, as it provides a uniform bound on the relaxation coefficient
$\beta'(S)=b(S)$, enabling the control of the discrete time derivative of $\beta(S)$ in Proposition~\ref{prop.beta}.

\begin{lemma}\label{lem.beta}
Let $\gamma_0>\gamma_2$. Then the function $\beta'=b$, defined in Assumption (A3), is bounded in $[0,1]$.
\end{lemma}

\begin{proof}
Set $r(S)=S^{\gamma_0}(1-S)^{-\gamma_1}$. Then
\begin{align*}
  b(S) = \frac{r(S)+S^{\gamma_0-\gamma_2}}{r(S)+1}
  \quad\mbox{for }0\le S<1.
\end{align*}
A computation shows that
\begin{align*}
  \frac{r'(S)}{r(S)} = \frac{\gamma_0}{S} + \frac{\gamma_1}{1-S} > 0,
\end{align*}
so $r(S)$ is increasing. Another calculation yields
\begin{align*}
  b'(S) = \frac{1}{(r(S)+1)^2}\big((\gamma_0-\gamma_2)
  S^{\gamma_0-\gamma_2-1}(r(S)+1) + (1-S^{\gamma_0-\gamma_2})r'(S)\big)
  > 0,
\end{align*}
where we use $\gamma_0-\gamma_2>0$. Thus, $b(S)$ is increasing, and since $b(0)=0$ and $b(1)=1$, we conclude the boundedness of $b$ in $[0,1]$.
\end{proof}

\begin{proposition}[Uniform bounds for the discrete time derivative of $\beta(S^k)$]\label{prop.beta}
Let Assumptions (A1)--(A4) hold. Then there exists a constant $C>0$ depending on $\bm{S}^0$ such that
\begin{align}\label{5.est2}
  \Delta t \sum_{k=1}^N \sum_{K\in\T}
  \m(K)\bigg|\frac{\beta(S_K^k)-\beta(S_K^{k-1})}{\Delta t}\bigg|^2
  + \Delta t \sum_{k=1}^N \sum_{\sigma\in\E}\tau_\sigma a_\sigma^k
  \bigg|\D_{K,\sigma}\bigg(\frac{\beta(S^k)-\beta(S^{k-1})}{\Delta t}
  \bigg)\bigg|^2 \le C.
\end{align}
\end{proposition}

\begin{proof}
We sum equations \eqref{2.eq1} over $i=1,\ldots,n$. Taking into account that $\sum_{i=1}^n M_{ij}(\bm{S}^k)=0$, this yields
\begin{align*}
  \frac{\m(K)}{\Delta t}\Phi_K(S_K^k-S_K^{k-1})
  = \sum_{i=1}^n\sum_{\sigma\in\E_K}\tau_\sigma
  c_{i,\sigma}^k a_\sigma^k\D_{K,\sigma}\bigg(P_c(S^k)
  + \frac{\beta(S^k)-\beta(S^{k-1})}{\Delta t}\bigg).
\end{align*}
We multiply this equation by $(\beta(S_K^k)-\beta(S_K^{k-1}))/\Delta t$, sum over $K\in\T$, and use discrete integration by parts \eqref{2.dibp}. This gives $J_1+J_2=0$, where
\begin{align*}
  J_1 &= \frac{1}{(\Delta t)^2}\sum_{K\in\T}
  \m(K)\Phi_K(S_{K}^k-S_{K}^{k-1})
  \big(\beta(S_K^k)-\beta(S_K^{k-1})\big), \\
  J_2 &= \frac{1}{\Delta t}\sum_{i=1}^n\sum_{\sigma\in\E}
  \tau_\sigma c_{i,\sigma}^k a_\sigma^k\D_{K,\sigma}\bigg(P_c(S^k)
  + \frac{\beta(S^k)-\beta(S^{k-1})}{\Delta t}\bigg)
  \D_{K,\sigma}\big(\beta(S^k)-\beta(S^{k-1})\big).
\end{align*}
Since $\beta'$ is bounded by Lemma \ref{lem.beta}, the first term can be estimated as
\begin{align*}
  J_1 &= \frac{1}{(\Delta t)^2}\sum_{K\in\T}\m(K)\Phi_K
  \frac{S_{K}^k-S_{K}^{k-1}}{\beta(S_K^k)-\beta(S_K^{k-1})}
  \big(\beta(S_K^k)-\beta(S_K^{k-1})\big)^2 \\
  &\ge \frac{C(\Phi,\beta')}{(\Delta t)^2}
  \sum_{K\in\T}\m(K)\big(\beta(S_K^k)-\beta(S_K^{k-1})\big)^2,
\end{align*}
where $C(\Phi,\beta') = c_\Phi/\|\beta'\|_{L^\infty(0,1)}$. We deduce from $\sum_{i=1}^n c_{i,\sigma}^k=1$ and Young's inequality that
\begin{align*}
  J_2 &= \frac{1}{\Delta t}\sum_{\sigma\in\E}\tau_\sigma a_\sigma^k
  \D_{K,\sigma}\bigg(P_c(S^k)
  + \frac{\beta(S^k)-\beta(S^{k-1})}{\Delta t}\bigg)
  \D_{K,\sigma}\big(\beta(S^k)-\beta(S^{k-1})\big) \\
  &= \frac{1}{\Delta t}\sum_{\sigma\in\E}
  \tau_\sigma a_\sigma^k\D_{K,\sigma}P_c(S^k)
  \D_{K,\sigma}\big(\beta(S^k)-\beta(S^{k-1})\big) \\
  &\phantom{xx}+ \frac{1}{(\Delta t)^2}
  \sum_{\sigma\in\E}\tau_\sigma a_\sigma^k
  \big|\D_{K,\sigma}\big(\beta(S^k)-\beta(S^{k-1})\big)\big|^2 \\
  &\ge \frac{1}{2(\Delta t)^2}\sum_{\sigma\in\E}\tau_\sigma a_\sigma^k
  \big|\D_{K,\sigma}\big(\beta(S^k)-\beta(S^{k-1})\big)\big|^2
  - \frac{1}{2}\sum_{\sigma\in\E}\tau_\sigma a_\sigma^k
  |\D_{K,\sigma}P_c(S^k)|^2.
\end{align*}
We infer from Lemma \ref{lem.aux} below that for $p_0\le\gamma_2$, $p_1\le\gamma_1$,
\begin{align*}
  \frac{P_c(S_L^k)-P_c(S_K^k)}{\psi(S_L^k)-\psi(S_K^k)} \le 1
  \quad\mbox{if }S_L^k\neq S_K^k.
\end{align*}
By the definition of $a_\sigma^k$, we obtain
\begin{align*}
  a_\sigma^k|\D_{K,\sigma}P_c(S^k)|^2
  &= \frac{\beta(S_L^k)-\beta(S_K^k)}{\psi(S_L^k)-\psi(S_K^k)}
  (P_c(S_L^k)-P_c(S_K^k))^2 \\
  &\le (\beta(S_L^k)-\beta(S_K^k))(P_c(S_L^k)-P_c(S_K^k)).
\end{align*}
We conclude that
\begin{align*}
  J_2 &\ge \frac{1}{2(\Delta t)^2}\sum_{\sigma\in\E}
  \tau_\sigma a_\sigma^k
  \big|\D_{K,\sigma}\big(\beta(S^k)-\beta(S^{k-1})\big)\big|^2
  - \frac{1}{2}\sum_{\sigma\in\E}\tau_\sigma
  \D_{K,\sigma}\beta(S^k)\D_{K,\sigma}P_c(S^k).
\end{align*}

Finally, we combine the estimates of $J_1$ and $J_2$, multiply them by $\Delta t$, and sum over $k=1,\ldots,N$:
\begin{align*}
  C(&\Phi,\beta')\Delta t\sum_{k=1}^N\sum_{K\in\T}
  \m(K)\bigg(\frac{\beta(S_K^k)-\beta(S_K^{k-1})}{\Delta t}\bigg)^2 \\
  &\phantom{xx}+ \frac{\Delta t}{2}\sum_{k=1}^N
  \sum_{\sigma\in\E}\tau_\sigma a_\sigma^k
  \bigg|\D_{K,\sigma}\bigg(\frac{\beta(S^k)-\beta(S^{k-1})}{\Delta t}
  \bigg)\bigg|^2 \\
  &\le C\Delta t\sum_{k=1}^N\sum_{\sigma\in\E}\tau_\sigma
  \D_{K,\sigma}\beta(S^k)\D_{K,\sigma}P_c(S^k) \le C,
\end{align*}
where the last inequality follows from the discrete energy inequality in Theorem \ref{thm.dei}, summed over $k=1,\ldots,N$.
\end{proof}

\begin{lemma}\label{lem.aux}
Assume that $p_0\le \gamma_2$ and $p_1\le \gamma_1$. Then for all $S$, $S'\in(0,1)$ with $S\neq S'$,
\begin{align*}
  0<\frac{P_c(S)-P_c(S')}{\psi(S)-\psi(S')}\le 1.
\end{align*}
\end{lemma}

\begin{proof}
By Assumption (A4) and the conditions $p_0\le\gamma_2$, $p_1\le\gamma_1$, we have for $0<z<1$,
\begin{align*}
  0 < P_c'(z) = z^{-p_0} + (1-z)^{-p_1} \le z^{-\gamma_2} + (1-z)^{-\gamma_1} = \psi'(z).
\end{align*}
Hence, $P_c$ and $\psi$ are strictly increasing, and $\psi-P_c$ is nondecreasing in $(0,1)$. Since the quotient is symmetric in $S$ and $S'$, we may assume that $S'<S$. Then $0<P_c(S)-P_c(S')\le\psi(S)-\psi(S')$, which proves the claim.
\end{proof}

\begin{remark}[Convergence of the scheme]\label{rem.conv}\rm
The convergence of the scheme to a weak solution as $\eta\to 0$ is expected to follow the usual strategy. The discrete energy inequality yields uniform bounds for $\beta(S)$ and $c_i$ in the discrete $L^2(0,T;H^1)$ space and for $\pa_t^\Delta S_i$, where $\pa_t^\Delta s=(s^k-s^{k-1})/\Delta t$, in the discrete dual space. A discrete nonlinear Aubin--Lions argument \cite{ACM17,GaLa12} then gives strong convergence of $S$ and $c_i$, and the passage to the limit in the flux terms is obtained by products of strongly and weakly convergent sequences, using discrete gradients on the diamond cells as in \cite{CLP03,JuZu21}. The difficulty lies in the identification of the limit of the pseudo-parabolic term $c_{i,\sigma}^k a_\sigma^k\,\D_{K,\sigma}\pa_t^\Delta\beta(S^k)$. Because $a(S)$ vanishes at $S\in\{0,1\}$, the energy estimate controls only the weighted gradient $\sqrt{a}\,\na\pa_t^\Delta\beta(S)$, which converges weakly (up to a subsequence) to some $\zeta$. The limit equation requires the identification $\zeta=\sqrt{a(S)}\,\na\pa_t\beta(S)$, which is a challenging point, since the degeneracy of $a$ prevents a uniform bound on $\na\pa_t^\Delta\beta(S)$. We therefore leave the complete convergence proof open.
\end{remark}

\section{Numerical experiments}\label{sec.exp}

In this section, we present several 2D simulations to evaluate how the thermodynamically consistent cross-diffusion model behaves physically, while also testing the stability of our finite-volume scheme. We handle the time integration via a fully implicit Euler method and resolve the nonlinear systems at each step using standard Newton iterations.

All simulations take place on a unit square domain $\Omega = (0,1)^2$
with a constant porosity $\Phi \equiv 1$. We use a simple Cartesian mesh made of $50 \times 50$ control volumes (i.e.\ $N_x = N_y = 50$). We impose the following boundary and initial conditions: $S_i^D=0.1$ ($i=1,2,3$) on $\Gamma_D=\{1\}\times(0,1)$, no-flux conditions on $\Gamma_N=\pa\Omega\setminus\Gamma_D$, and $S_i^0=0.3$ in the rectangle $R_i$ and $S_i^0=0.1$ elsewhere, where $R_1=[0.2,0.4)\times[0.3,0.7]$, $R_2=[0.4,0.6)\times[0.3,0.7]$, and $R_3=[0.6,0.8]\times[0.3,0.7]$. The time step is $\Delta t=2.5\cdot10^{-3}$, and the Newton iteration is stopped when the Euclidean norm of the residual is smaller than $10^{-8}$.

Following Assumptions (A3) and (A4), we set the parameters for the mobility $a(S)$, the relaxation $b(S)$, and the static capillary pressure $P_c(S)$ to $\gamma_0 = 4.0$, $\gamma_1 = 3.0$, $\gamma_2 = 3.0$, and $p_0 = p_1 = 2.1$. To couple the $n=3$ species, we use the mobility matrix
$$
M(\bm{S}) =
\begin{pmatrix}
\frac{\alpha+S_1 S_2}{f_{12}} + \frac{\alpha+S_1 S_3}{f_{13}} & -\frac{\alpha+S_1 S_2}{f_{12}} & -\frac{\alpha+S_1 S_3}{f_{13}} \\[4pt]
-\frac{\alpha+S_1 S_2}{f_{12}} & \frac{\alpha+S_1 S_2}{f_{12}} + \frac{\alpha+S_2 S_3}{f_{23}} & -\frac{\alpha+S_2 S_3}{f_{23}} \\[4pt]
-\frac{\alpha+S_1 S_3}{f_{13}} & -\frac{\alpha+S_2 S_3}{f_{23}} & \frac{\alpha+S_1 S_3}{f_{13}} + \frac{\alpha+S_2 S_3}{f_{23}}
\end{pmatrix},
$$
where the coefficients $f_{ij}>0$ are friction coefficients between species $i$ and $j$ (a small value of $f_{ij}$ means a strong coupling) and $\alpha>0$ is a small regularization parameter. To describe a highly heterogeneous mixing process, we choose $f_{12}=f_{23}=1$ and $f_{13}=0.01$. The regularization parameter is set to $\alpha=10^{-3}$. The matrix $M(\bm{S})$ is smooth, symmetric, satisfies $\sum_{i=1}^3 M_{ij}(\bm{S})=0$, and fulfills Assumption (A2). Indeed, $z^TM(\bm{S})z=\sum_{i<j}w_{ij}(z_i-z_j)^2$ with $w_{ij}=(\alpha+S_iS_j)/f_{ij}$, $\sum_{i<j}(z_i-z_j)^2=3|\Pi z|^2$ for $z\in\mathbb{R}^3$, and $S_iS_j\le 1/4$ in $\dom$, so that (A2) holds with
$$
  c_M=\frac{3\alpha}{\max_{i,j}f_{ij}}, \qquad C_M=\frac{3(\alpha+1/4)}{\min_{i,j}f_{ij}}.
$$
For $\alpha=0$ and equal coefficients $f_{ij}=f$, the matrix reduces to $(S^2/f)(\operatorname{diag}(c)-c\otimes c)$ with $c_i=S_i/S$, which is, up to the factor $S^2$, the Maxwell--Stefan mobility matrix with equal binary diffusion coefficients $D_{ij}=1/f$. In general, $M$ is a regularized friction (Onsager-type) approximation of the Maxwell--Stefan mobility. The regularization is needed for the lower bound in (A2): For $\alpha=0$, we have $z^TM(\bm{S})z\to 0$ for $z=(1,0,0)$ as $S_1\to 0$, although $\Pi z\neq 0$.

A key goal of our tests is to see how the dynamic capillarity changes the transient mixing phase. For this purpose, we compare the complete model that includes the relaxation term ($\partial_t \beta \neq 0$) and a simplified classical version with $\partial_t \beta = 0$. In the latter case, we only drop the term $(\beta(S^k)-\beta(S^{k-1}))/\Delta t$ in the numerical flux, while $a_\sigma^k$ from \eqref{2.quot} and the chemical potentials $\mu_i^k$ are kept unchanged. Its energy is the free energy $\sum_{K\in\T}\m(K)\Phi_K E(\bm{S}_K^k)$, i.e.\ $H_d$ without the interfacial term, which is nonincreasing by the same arguments as in the proof of \eqref{1.dei}. This classical model, however, still contains the capillary energy $\Psi$ (defined via $b/a$) in the free energy and in the chemical potentials, so it differs from the standard multicomponent Richards model with $b=0$.

The evolution of the mixture under dynamic capillary pressure is shown in Figure \ref{fig:caso8}, whereas Figure \ref{fig:caso9} reports the results for the classical system. The figures show the saturations $S_1$ (upper row), $S_2$ (middle row), and $S_3$ (lower row) at various times. The tiny value of $f_{13}$ triggers intense cross-diffusion, causing the species to mix very quickly and reach a nearly homogeneous state around $t=0.3$. The presence of the dynamic capillary pressure term, i.e.\ $\partial_t\beta \neq 0$, dampens the effect of capillary diffusion without affecting the cross-diffusion.

\begin{figure}[htbp]
    \centering \includegraphics[width=0.85\textwidth]{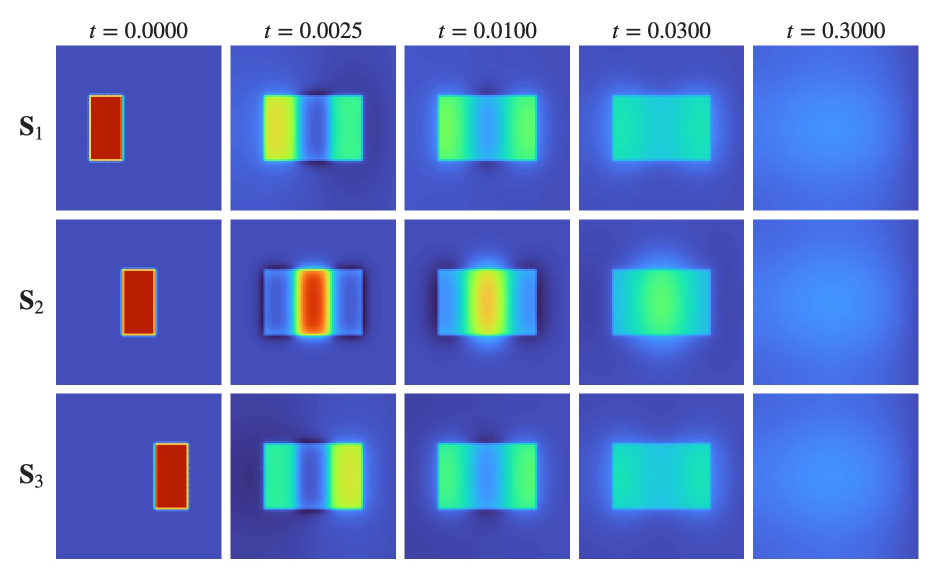}
    \caption{Time evolution of the three-species mixture incorporating dynamic capillary pressure ($\partial_t \beta \neq 0$) {at times $t=0,0.0025,0.01,0.03,0.3$ (from left to right)}.}
    \label{fig:caso8}
\end{figure}

\begin{figure}[htbp]
    \centering
    \includegraphics[width=0.85\textwidth]{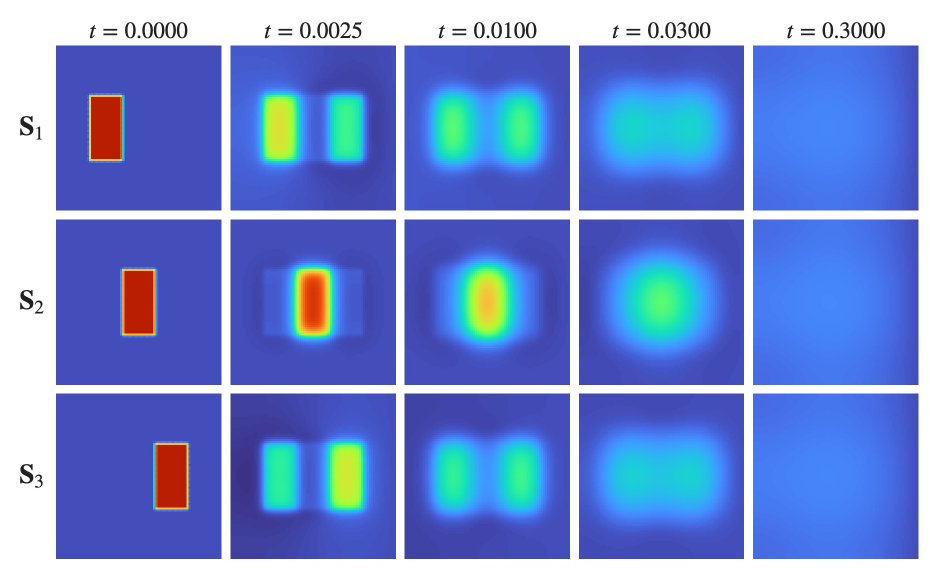}
    \caption{Time evolution of the three-species mixture without dynamic capillary pressure ($\partial_t \beta = 0$) {at times $t=0,0.0025,0.01,0.03,0.3$ (from left to right)}.}
    \label{fig:caso9}
\end{figure}

To assess thermodynamic consistency, we compute the discrete total energy $H_d(\bm{S}^k)$ from \eqref{1.Hd} (which reduces to the free energy when $\partial_t\beta=0$) up to $T = 5$, allowing the system sufficient time to reach equilibrium. Figure \ref{fig:energy} shows this energy decay on a semi-logarithmic scale. As our theory predicts, $k\mapsto H_d(\bm{S}^k)$ decreases monotonically, in accordance with the discrete energy inequality \eqref{1.dei}. For $\partial_t\beta\neq 0$, the initial value is larger because of the interfacial term $\frac12\sum_{\sigma\in\E}\tau_\sigma|\D_{K,\sigma}\beta(S^0)|^2$, which is large since $S^0$ is discontinuous. Strictly speaking, such initial data do not satisfy $H(\bm{S}^0)<\infty$ in (A1), but $H_d(\bm{S}^0)$ is finite on every fixed mesh. The fact that energy dissipates without any numerical oscillations is a strong indicator of the stability of our scheme. In all simulations, we observe $S_{i,K}^k\ge 0.08$ and $0.3\le S_K^k\le 0.5$, confirming the positivity and boundedness of the discrete solutions.

\begin{figure}[htbp]
    \centering \includegraphics[width=0.75\textwidth,height=70mm]{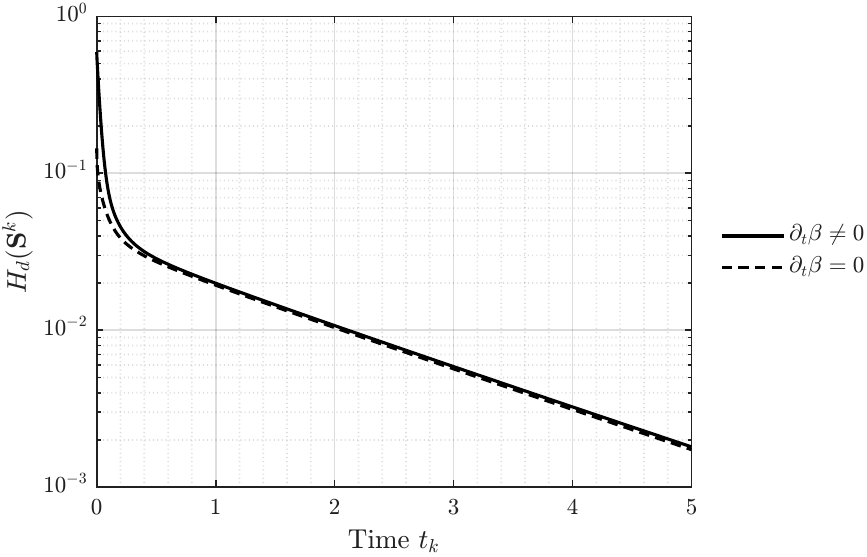}
    \caption{Evolution of the discrete total energy $H_d(\bm{S}^k)$ on a semilogarithmic scale.}
    \label{fig:energy}
\end{figure}

Finally, we run a Cauchy convergence test to assess the spatial accuracy of our finite-volume method for the model with dynamic capillary pressure. To avoid artificially enhanced convergence rates associated with late-time steady states, we evaluate the $L^1$ error at an intermediate simulation time, $T=0.2$, thereby capturing the strongly nonlinear transient dynamics. We use a highly refined $240 \times 240$ grid as the reference and calculate errors on meshes with $N=20, 30, 40, 60$, and $120$. The reference solution at the cell centers of the coarser meshes is taken as the average of the four adjacent cells of the reference mesh. All meshes use the same time step $\Delta t=2\cdot 10^{-4}$, so the reference solution contains the same time discretization error, and the computed differences measure the spatial discretization error for fixed $\Delta t$.

As shown in Figure \ref{fig:convergence}, the $L^1$ error for $S_1$ drops with an empirical rate of roughly $r \approx 0.99$, paralleling the $\mathcal{O}(\Delta x)$ reference triangle. This agrees with the expected first-order accuracy of two-point flux approximations: Our scheme prioritizes robust structure preservation---specifically through the logarithmic mean and the quotient \eqref{2.quot} in the two-point flux approximation---which successfully handles the degenerate mobility $a(S)$ and the cross-diffusion terms without spurious oscillations, resulting in the expected first-order spatial convergence.

\begin{figure}[htbp]
    \centering
    \includegraphics[width=0.75\textwidth,height=80mm]{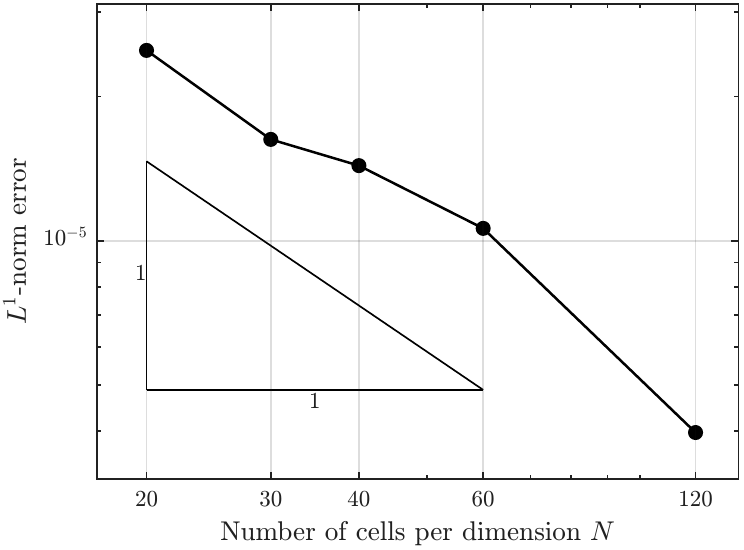}
    \caption{Spatial Cauchy convergence study in the $L^1$ norm for the saturation $S_1$ at time $T=0.2$. The scheme exhibits first-order accuracy $\mathcal{O}(\Delta x)$.}
    \label{fig:convergence}
\end{figure}


\begin{appendix}
\section{Some estimates for the continuous problem}\label{sec.app}

We show the energy inequality and a bound for $\pa_t\beta(S)$ for smooth solutions. Recall definition \eqref{1.E} of the free energy.

\begin{proposition}[Energy inequality]
Let $\bm{S}$ be a smooth solution to \eqref{1.Si}--\eqref{1.mui}. Then inequality \eqref{1.ei} holds for $t>0$.
\end{proposition}

\begin{proof}
Observing that $\pa E/\pa S_i = \mu_i$, we obtain
\begin{align*}
  \frac{d}{dt}\int_\Omega \bigg(\Phi(x)E(\bm{S})
  &+ \frac12|\na\beta(S)|^2\bigg)dx
  = \int_\Omega\bigg(\sum_{i=1}^n\Phi(x)\pa_t S_i\mu_i
  + \na\beta(S)\cdot\na\pa_t\beta(S)\bigg)dx \\
  &= -\int_\Omega\frac{a(S)}{S}\na\big(P_c(S)+\pa_t\beta(S)\big)
  \cdot\sum_{i=1}^n S_i\na\mu_i dx \\
  &\phantom{xx}
  - \sum_{i,j=1}^n\int_\Omega M_{ij}(\bm{S})\na\mu_j\cdot\na\mu_i dx
 + \int_\Omega\na\beta(S)\cdot\na\pa_t\beta(S)dx.
\end{align*}
The properties $\sum_{i=1}^n S_i=S$ and $\psi'=b/a=\beta'/a$ lead to
\begin{align*}
  \sum_{i=1}^n S_i\na\mu_i = \sum_{i=1}^n S_i
  \bigg(\frac{\na S_i}{S_i}-\frac{\na S}{S}\bigg)
  + \sum_{i=1}^n S_i\na\psi(S) = S\na\psi(S)
  = \frac{S}{a(S)}\na\beta(S).
\end{align*}
Then the terms involving $\pa_t\beta(S)$ cancel, and the identity
\begin{align*}
  a(S)\na {P_c(S)}\cdot\na\psi(S) = P_c'(S)\beta'(S)|\na S|^2,
\end{align*}
finishes the proof.
\end{proof}

\begin{proposition}[Bound for $\pa_t\beta(S)$]
Let $\bm{S}$ be a smooth solution to \eqref{1.Si}--\eqref{1.mui}. Assume that $c_\beta=1/\|\beta'\|_{L^\infty(0,1)}>0$ and that there exists $C>0$ such that $P_c'(S)\le C \psi'(S)$ for all $0<S<1$. Then, for a constant $C_1>0$ depending on $\bm{S}^0$,
\begin{align*}
    c_\beta c_\Phi\int_0^T \int_\Omega|\pa_t\beta(S)|^2 dx dt
    + \frac12\int_0^T\int_\Omega a(S)|\na\pa_t\beta(S)|^2 dx dt \le C_1.
\end{align*}
\end{proposition}

\begin{proof}
We sum equation \eqref{1.Si} over $i=1,\ldots,n$. We deduce from $\sum_{i=1}^n S_i/S=1$ and $\sum_{i=1}^n M_{ij}=0$ that
\begin{align*}
  \Phi(x)\pa_t S = \diver\big(a(S)\na(P_c(S)+\pa_t\beta(S))\big)
  \quad\mbox{in }\Omega,\ t>0.
\end{align*}
Multiplying this equation by $\pa_t\beta(S)$, integrating over $\Omega$, and integrating by parts, we find that
\begin{align*}
  \int_\Omega\Phi(x)\pa_t S\pa_t\beta(S) dx
  = -\int_\Omega a(S)\na\big(P_c(S)+\pa_t\beta(S)\big)
  \cdot\na\pa_t\beta(S)dx.
\end{align*}
The left-hand side is estimated as
\begin{align*}
  \int_\Omega\Phi(x)\pa_t S\pa_t\beta(S) dx
  = \int_\Omega\Phi(x)\frac{|\pa_t\beta(S)|^2}{\beta'(S)} dx
  \ge c_\beta c_\Phi\int_\Omega|\pa_t\beta(S)|^2 dx.
\end{align*}
Together with Young's inequality, this shows that
\begin{align*}
  c_\beta c_\Phi\int_\Omega|\pa_t\beta(S)|^2 dx
  &\le -\int_\Omega a(S)P_c'(S)\na S\cdot\na\pa_t\beta(S)dx
  - \int_\Omega a(S)|\na\pa_t\beta(S)|^2 dx \\
  &\le \frac12\int_\Omega a(S)P_c'(S)^2|\na S|^2 dx
  - \frac12\int_\Omega a(S)|\na\pa_t\beta(S)|^2 dx.
\end{align*}
By assumption, $P_c'(S)\le C\psi'(S) = Cb(S)/a(S) = C\beta'(S)/a(S)$ and hence
\begin{align*}
  c_\beta c_\Phi\int_\Omega|\pa_t\beta(S)|^2 dx
  + \frac12\int_\Omega a(S)|\na\pa_t\beta(S)|^2 dx
  \le \frac{C}{2}\int_\Omega \beta'(S)P_c'(S)|\na S|^2 dx.
\end{align*}
Integrating over the time and noting that the right-hand side is bounded by the energy inequality, ends the proof.
\end{proof}

\end{appendix}


\end{document}